\documentclass[12pt]{article}
\usepackage[intlimits]{amsmath}
\usepackage{amsfonts,amssymb,amsthm,wasysym,makecell,listings,authblk,hyperref}
\usepackage[utf8]{inputenc}
\usepackage[english]{babel}
\usepackage{graphicx}
\usepackage[table,dvipsnames]{xcolor}

\newtheorem{theorem}{Theorem}
\newtheorem{definition}{Definition}

\usepackage{multirow,longtable,indentfirst,hhline,array,listings,upquote}

\lstdefinelanguage{clingo}{
    morekeywords={not,level,A,B,C,D,E,F,R,X,Z,I,N,N1,N2,N3},
    otherkeywords={:-,..,\#count,\#sum,\#heuristic,\#show,\#const,!=,==},
    sensitive=true,
    morecomment=[l]{\%},
    morestring=[b]",
}

\usepackage[a4paper, textwidth=185mm, textheight=260mm, top=20mm, left=15mm, footskip=10mm]{geometry}

\renewcommand{\le}{\leqslant}
\renewcommand{\ge}{\geqslant}

\makeatletter
\newcommand{\draw@points}[2]{%
  \expandafter\ifx\familydefault#2\familydefault \else 
  \put(#1,#2){\@shape}%
  \expandafter\draw@points 
  \fi
}

\newcommand{\coordinates}[1]{
  \begin{center}
    \begin{minipage}{0.85\textwidth}
      \centering\footnotesize\sloppy
      #1
    \end{minipage}
  \end{center}
}

\title{\bf Combinatorial Geometry of Erdős--Szekeres Type Problems: SAT/ASP Modeling and Linear Subreduction}

\author{Vitalii Koshelev~~~ Alexey Koshka}
\affil{Independent Researchers, e-mail: koshelev@mccme.ru, biocheshire@yandex.ru}

\date{}

\begin{document}

\maketitle

\begin{abstract}
\footnotesize
This paper investigates several classical and novel variations of the Erdős--Szekeres problem, including multicolored point sets, convex hexagons with a given number of interior points, and polygons with constraints on edge colors. We propose a comprehensive computational framework combining combinatorial modeling within the SAT/ASP paradigms with the geometric realization of configurations. To determine point coordinates, we developed the \textbf{linear subreduction method}. The core idea consists of combining the complete logical model of the problem with a system of geometric inequalities, followed by fixing the abscissae to linearize the constraints. This approach enables a simultaneous search for a realization across the entire space of admissible abstract configurations (signotopes) rather than examining them individually, while linearization significantly accelerates the SMT solving process. Using this framework, we established new exact values for several functions; in particular, we proved $h_{nc}(4,0; 4,0)=26$: any bicolored set of 26 points in general position must contain the vertices of an empty monochromatic quadrilateral.

\end{abstract}


\section{Introduction and Statement of Results}

In 1935, Erdős and Szekeres formulated the following problem (see \cite{ES}, \cite{Low}).

\vskip+0.2cm
{\bf The First Erdős--Szekeres Problem.} {\it For any integer $n \ge 3$, find the smallest positive integer $g(n)$ such that any set of at least $g(n)$ points in the plane in general position contains a subset of $n$ points that form the vertices of a convex $n$-gon.}
\vskip+0.2cm

Recall that a set of points is in general position if no three of its elements are collinear.

In 1978, Erdős proposed a modification of the first problem (see \cite{E}).
\vskip+0.2cm
{\bf The Second Erdős--Szekeres Problem.} {\it For any integer $n \ge 3$, find the smallest positive integer $h(n)$ such that any set ${\cal X}$ of at least $h(n)$ points in the plane in general position contains a subset of $n$ points that form the vertices of a convex and empty $n$-gon, i.e., an $n$-gon containing no points of ${\cal X}$ in its interior.}
\vskip+0.2cm

These problems are classical in combinatorial geometry and Ramsey theory (see \cite{Sol}, \cite{Ram}, \cite{TRam}, \cite{Hall}). Both can be generalized as follows.

\vskip+0.2cm
{\bf The Third Erdős--Szekeres-type Problem.} {\it For any integers $n \ge 3$ and $k \ge 0$, find the smallest positive integer $h(n,k)$ such that any set ${\cal X}$ of at least $h(n,k)$ points in the plane in general position contains a subset $C$ of size $n$ forming the vertices of a convex $n$-gon $C$ with $|(\operatorname{conv}(C) \cap {\cal X}) \setminus C| \le k$, i.e., this $n$-gon contains at most $k$ other points of ${\cal X}$ in its interior.}
\vskip+0.2cm

Devillers, Hurtado, Károlyi, and Seara \cite{chrom} proposed a generalization of the first two problems by considering multicolored point sets in the plane. In our work, we study the following formulation.

\vskip+0.2cm
{\bf Erdős--Szekeres-type Problem for Bicolored Sets.} {\it For any integers $n_1 \ge 3, k_1 \ge 0, n_2 \ge 3$, and $k_2 \ge 0$, find the smallest positive integer $h(n_1,k_1;n_2,k_2)$ such that any bicolored set ${\cal X}$ of at least $h(n_1,k_1;n_2,k_2)$ points in general position contains either a subset of size $n_1$ forming a convex $n_1$-gon of the first color with at most $k_1$ interior points, or a subset of size $n_2$ forming a convex $n_2$-gon of the second color with at most $k_2$ interior points. We also define $h_{nc}(n_1,k_1;n_2,k_2)$ for the case where the convexity condition is not required.}
\vskip+0.2cm

While we assumed $n_1, n_2 \ge 3$, the problem remains well-defined for $n_1=2$ or $n_2=2$, where a 2-gon is simply a pair of points. Finally, one can consider an arbitrary number of colors, defining the values $h(n_1,k_1;n_2,k_2;n_3,k_3;\dots)$ and $h_{nc}(n_1,k_1;n_2,k_2;n_3,k_3;\dots)$.

Additionally, we define $h_{isl}(n_1,k_1;n_2,k_2;n_3,k_3;\dots)$ for monochromatic $n$-islands with at most $k$ interior points, where an $n$-island is an arbitrary subset $C \subset {\cal X}$ such that $|C|=n$ and $|(\operatorname{conv}(C) \cap {\cal X}) \setminus C| \le k$.

For convenience, we introduce the generalized notation $h_*(\cdots)$ to refer to any of the above functions ($h, h_{nc}, h_{isl}$) unless the context dictates otherwise.

For the first Erdős--Szekeres problem, it is known that:
$$g(3)=3, \quad g(4)=5, \quad g(5)=9 \text{ \cite{ES}}, \quad g(6)=17 \text{ \cite{SL}}.$$

The latter result was first proved by Szekeres and Peters in 2006 using an exhaustive computer search that took 1500 hours. In 2019--2020, Marić \cite{maric} and Scheucher \cite{scheucher2020} applied SAT solvers to reduce the verification time to one hour.

For arbitrary $n$, the upper bound has been repeatedly improved. We have summarized all known results in the following table. Ultimately:
$$ \text{\cite{Low} } 2^{n-2}+1 \le g(n) \le 2^{n + O(\sqrt{n \log n})} \text{ \cite{HMPT20}}. $$

\begin{table}[h!]
\centering
\footnotesize
\renewcommand{\arraystretch}{1.5}
\begin{tabular}{|l|c|c c c c c c c c c c|}
\hline
\textbf{Year, Authors} & \textbf{Formula} & \textbf{3} & \textbf{4} & \textbf{5} & \textbf{6} & \textbf{7} & \textbf{8} & \textbf{9} & \textbf{10} & \textbf{11} & \textbf{12} \\ \hline
1935, Erdős--Szekeres \cite{ES} & $\binom{2n-4}{n-2} + 1$ & 3 & 7 & 21 & 71 & 253 & 925 & 3433 & 12871 & 48621 & 184757 \\
1998, Chung--Graham \cite{CG98} & $\binom{2n-4}{n-2}$ & & 6 & 20 & 70 & 252 & 924 & 3432 & 12870 & 48620 & 184756 \\
1998, Kleitman--Pachter \cite{KP98} & $\binom{2n-4}{n-2} - 2n + 7$ & 3 & 5 & 17 & 65 & 245 & 915 & 3421 & 12857 & 48605 & 184739 \\
1998, Tóth--Valtr \cite{TV98} & $\binom{2n-5}{n-2} + 2$ & & 5 & 12 & 37 & 128 & 464 & 1718 & 6437 & 24312 & 92380 \\
2005, Tóth--Valtr \cite{TV05} & $\binom{2n-5}{n-2} + 1$ & & & 11 & 36 & 127 & 463 & 1717 & 6436 & 24311 & 92379 \\
2015, Vlachos \cite{Vl15} & $(\frac{29}{32}+o(1)) \binom{2n-5}{n-2}$ & & & & 33 & 114 & 414 & 1536 & 5765 & 21804 & 82942 \\
2015, Mojarrad--Vlachos \cite{MV15} & $\binom{2n-5}{n-2}-\binom{2n-8}{n-3}+2$ & & & 11 & 33 & 113 & 408 & 1508 & 5645 & 21309 & 80490 \\
2015, Norin--Yuditsky \cite{NY15} & $(\frac{7}{8}+o(1)) \binom{2n-5}{n-2}$ & & & 10 & 28 & 92 & 324 & 1178 & 4358 & 16304 & 61492 \\
2017, Suk \cite{Suk17} & $2^{n + O(n^{2/3}\log n)}$ & \multicolumn{10}{c|}{\textit{asymptotic bound}} \\
2020, Holmsen et al. \cite{HMPT20} & $2^{n + O(\sqrt{n \log n})}$ & \multicolumn{10}{c|}{\textit{asymptotic bound}} \\
\textbf{1935, Conjecture} & $2^{n-2} + 1$ & \textbf{3} & \textbf{5} & \textbf{9} & \textbf{17} & \textbf{33} & \textbf{65} & \textbf{129} & \textbf{257} & \textbf{513} & \textbf{1025} \\ \hline
\end{tabular}
\caption{\footnotesize Comparison of known upper bounds for $g(n)$ for the Erdős--Szekeres convex polygon problem. Empty cells correspond to $n$ values outside the applicability of the methods.}
\end{table}

The second problem is completely solved. The following results have been proven:
$$h(3)=3, \quad h(4)=5, \quad h(5)=10 \text{ \cite{Harb}}.$$

The history of finding the value of $h(6)$ evolved from early analytical investigations into rigorous formal verification. The first lower bounds were obtained through computer search: in 1989, Overmars, Scholten, and Vincent \cite{Over0} presented configurations of 26 points without empty hexagons. In 2003, Overmars \cite{Over}, using an optimized algorithm for finding empty polygons, improved this result by constructing a set of 29 points, establishing the bound $h(6) \ge 30$.

The question of whether $h(6)$ is finite remained open for nearly thirty years until independent proofs of existence were published in 2007--2008. Nicolas \cite{Nicolas} presented a proof with an upper bound of $h(6) \le g(25)$, and Gerken \cite{Gerken} with $h(6) \le g(9)$. Soon after, Valtr \cite{Valtr} proposed a simplified version of Gerken's argument, giving $h(6) \le g(15)$. However, these values were orders of magnitude larger than the known lower bound\footnote{In 2007, one of the authors announced an estimate $h(6) \le g(8)$. The work was planned in two parts, the first of which was published in 2009 \cite{MSb}. However, during the preparation of the second part, technical gaps in the proof were identified. Since the exact value $h(6) = 30$ has now been established by computational methods, the publication of the corrected second part was deemed unnecessary.}.

The exact value was only established in 2024 by Scheucher and Heule \cite{h6}: by combining geometric theory with the power of SAT solvers (CaDiCaL \cite{cadical}), they proved that any set of 30 points in general position contains a convex empty hexagon. The definitive result on this matter was given in the work by Subercaseaux et al. \cite{h6_formal}. The authors presented a full formal verification of Scheucher and Heule's result in the Lean 4 system, mathematically proving the correctness of the geometric encoding and verifying the logical inference. This elevated the result from the category of computer calculations to the status of a formally proven theorem.

For $n \ge 7$, the value $h(n)$ does not exist, as proven by Horton in 1983 (see \cite{Hort}). This leads to the third Erdős--Szekeres type problem and, in particular, the question of the existence of $h(n,k)$ for $n > 7$.

For this problem, the obvious inequalities $g(n) \le h(n,k) \le h(n)$ hold if the corresponding values exist. Furthermore:
$$h(n) = h(n, 0) \ge h(n, 1) \ge h(n, 2) \ge \dots \ge g(n).$$
There exists a maximum value $\overline{k} = \overline{k}(n)$ such that $h(n, \overline{k}) > g(n)$, while for all $k > \overline{k}$, $h(n, k) = g(n)$ (for example, it is obvious that $h(n, g(n)-n) = g(n)$).

For small values of $n$, the following results are known:
$$h(3, k) = 3, \quad h(4, k) = 5, \quad h(5, 0) = 10, \quad h(5, k \ge 1) = 9.$$
The last equality is due to the fact that any convex pentagon containing two or more points of the set in its interior always contains a smaller convex and empty pentagon.

Deeper results on the third problem were obtained in the works of Sendov \cite{Sen} and Nyklova \cite{Nyk}. Using Horton constructions \cite{Hort}, these papers prove the non-existence of $h(n, k)$ for certain values of $k$ when $n > 7$. If $\underline{k}(n)$ denotes the maximum value of $k$ for which $h(n, k)$ does not exist for a given $n$, Sendov and Nyklova obtained the estimate $\underline{k}(n) \ge (\sqrt[4]{2} + o(1))^n$. One of the authors established a significantly stronger exponential estimate: $\underline{k}(n) \ge (2 + o(1))^n$ (see \cite{Kosh}). Specifically, it was shown that for odd $n$, the following does not exist:
$$ h\left(n, \binom{n-7}{\frac{n-7}{2}} - 1\right), $$
and for even $n$:
$$ h\left(n, 2\binom{n-8}{\frac{n-8}{2}} - 1\right). $$

In Nyklova's paper \cite{Nyk}, it was also proven that $h(6,6) = g(6) = 17$ and the estimate $h(6,5) \le 19$ was obtained\footnote{In several prior publications, including those of one of the authors, it was erroneously claimed that the result $h(6, 5) = 19$ in \cite{Nyk} was due to a computational error. This misunderstanding arose from the ambiguity of the author's notation: in the original work \cite{Nyk}, the value 19 is given only as an upper bound, although it is formally written with an equals sign. This clarification is intended to correct the tradition of incorrect citation and interpretation of this result in the literature.}.

In 2008, one of the authors showed that $h(6,1) \le g(7) \le 127$ (see \cite{FPM}). In 2010, using a modification of the Szekeres-McKay-Peters algorithm \cite{SL}, the same author found the exact values of $h(6,2)$ and $h(6,1)$ (see \cite{h61_comp}). In the present study, we present a new computer proof obtained by the method of Answer Set Programming (ASP). The verification time for these cases has been significantly reduced to 70 and 190 minutes, respectively.

\begin{theorem}\label{h61}
The following equalities hold: $h(6, \ge 2) = 17$, $h(6, 1) = 18$.
\end{theorem}

Now let us discuss the version of Erdős--Szekeres problem for bicolored sets. It is easy to see that $h(n_1,\infty;n_2,\infty)=g(n_1)+g(n_2)-1$.
Devillers et al. in \cite{chrom} proved that among $N\ge 5$ points of two colors, there must be $\lceil N/4\rceil-2$ monochromatic disjoint triangles (i.e., triangles intersecting at most by a common edge).
From this, it follows that:
$$h(3,0;3,0)=9, \quad h(3,0;3,0;2,0)=13+1=14,$$
$$h(3,0;3,0;2,0;2,0)=h(3,0;3,0;3,\infty)=h(3,0;3,0;3,8)=17+2=19,$$
$$h(3,0;3,0;\underbrace{2,0;\dots;2,0}_t)=5t+9.$$
Also, using Horton set colorings, the authors of \cite{chrom} proved that $h(3,0;3,0;3,0)$ and $h(5,0;5,0)$ do not exist. Furthermore, their proof implies that $h(3,0;5,0)$ does not exist.
However, $h(3,0;4,0)\le h(6)$ exists, since coloring the vertices of an empty hexagon in two colors always forms the required convex triangle or quadrilateral. The question of the existence of $h(4,0;4,0)$ and $h_{isl}(4,0;4,0)$ remains open.

Lower bounds for $h(4,0;4,0)$ have been repeatedly improved: from 18 points in \cite{chrom} to the results of Brass (20 points \cite{Brass_ref}), Friedman (30 points \cite{Friedman_ref}), van Gulik (32 points \cite{Gulik_ref}), Huemer and Seara (36 points \cite{HS_ref}). Finally, one of the authors of the present work established that $h(4,0;4,0)\ge 47$ by constructing an example with 46 points \cite{h46}.

In 2016, Basu et al. \cite{basu} found the values:
$$h(3,1;3,1)=6, \quad h(3,1;3,1;3,1)=13,$$
and proved that $h(3,c-1;\ldots;3,c-1)\le \max\{c^2+1,6\}$, $h(3,c-2;\ldots;3,c-2)\le c^2+c+1$, while the value $h(3,\lfloor\frac{c-3}{2}\rfloor;\ldots;3,\lfloor\frac{c-3}{2}\rfloor)$ does not exist.

In 2019, Cravioto-Lagos et al. \cite{cravioto} showed that:
$$h(3,c-3;\ldots;3,c-3)\le \left\lfloor\frac{2c(c+\frac{1}{c-2})}{\frac{c-2}{c-1}-\frac{2c-3}{(c-1)(c-2)^2}}\right\rfloor+1,$$
$$h(4,2c-3;\ldots;4,2c-3)\le c\cdot g(4c+1).$$

Due to the complexity of the existence problem for $h(4,0;4,0)$, Aichholzer et al. considered in 2010 its simplification (see \cite{nconv4}), where the convexity condition is not mandatory, and proved that $h_{nc}(4,0;4,0)\le 2760$. The example for the convex case from \cite{chrom}, consisting of 18 points, applies here as well.

The latest lower bounds for $h_{nc}(4,0;4,0)$ and $h_{isl}(4,0;4,0)$ were obtained within the EuroGIGA ComPoSe project under Aichholzer's leadership. Specifically, the computer search conducted by project participants allowed the construction of a set of 22 points for $h_{nc}(4,0;4,0)$ and an example of 35 points for $h_{isl}(4,0;4,0)$ \footnote{\url{https://www.eurogiga-compose.eu/posezo.php}}\footnote{In his Ph.D. thesis, Scheucher \cite{scheucher2020} claimed to have found examples of 48, 36, and 24 points for $h(4,0;4,0)$, $h_{isl}(4,0;4,0)$, and $h_{nc}(4,0;4,0)$ respectively. However, at the time of writing, the original web repository of the project \url{https://page.math.tu-berlin.de/~scheuch/research/sat_vs_bicolored_point_sets/} is unavailable, and the thesis does not provide full sets of coordinates, making independent verification of these results impossible.}.

We have constructed an example of a set with 25 points containing no empty monochromatic quadrilaterals and proved the absence of such configurations for $N\ge 26$, establishing the exact value $h_{nc}(4,0;4,0)=26$.

In 2018, Liu and Zhang \cite{liu} established that:
$$h_{nc}(4,2;4,2)=9, \quad h_{nc}(4,1;4,1)=11,$$
$$h_{nc}(4,2;4,2;4,2)\le 120, \quad h_{nc}(4,2c-3;\ldots;4,2c-3)\le 4c^2+1$$
and that the value $h_{nc}(4,2\lfloor\frac{c-1}{2}\rfloor-1;\ldots;4,2\lfloor\frac{c-1}{2}\rfloor-1)$ does not exist.

Note that monochromatic variants of the problem for $h_{nc}$ and $h_{isl}$ are trivial, since $h_{nc}(n,0)=h_{isl}(n,0)=n$.

The main results of this study are formulated below. The values of the functions $h_*(\cdots)$ obtained through computer modeling and analytical derivation are presented in the following statements.

\begin{theorem}
The values of the functions $h_{nc}(4,k_1;4,k_2)$, $h_{isl}(4,k_1;4,k_2)$, and $h(4,k_1;4,k_2)$ are given in the tables:
\begin{table}[h!]
\centering
\begin{tabular}{|c|cccc| c |c|cccc| c |c|ccccc|}
\hhline{-----~-----~------}
\multicolumn{5}{|c|}{$h_{nc}(4,k_1;4,k_2)$} & & \multicolumn{5}{c|}{$h_{isl}(4,k_1;4,k_2)$} & & \multicolumn{6}{c|}{$h(4,k_1;4,k_2)$} \\
\hhline{-----~-----~------}
& 0 & 1 & 2 & 3 & & & 0 & 1 & 2 & 3 & & & 0 & 1 & 2 & 3 & 4 \\
\hhline{-----~-----~------}
0 & 26 & 15 & 14 & 12 & & 0 & 36-& 22 & 16 & 13 & & 0 & 47-   & 29-& 23-25 & 20-21 & 17 \\
1 & 15 & 11 & 10 & 9  & & 1 & 22 & 13 & 11 & 9  & & 1 & 29-   & 18 & 16 & 15 & 13 \\
2 & 14 & 10 & 9 & 9   & & 2 & 16 & 11 & 9 & 9   & & 2 & 23-25 & 16 & 12 & 12 & 11 \\
3 & 12 & 9 & 9 & 7    & & 3 & 13 & 9 & 9 & 7    & & 3 & 20-21 & 15 & 12 & 11 & 11 \\
\hhline{-----~-----~}
\multicolumn{12}{c|}{}                            & 4 & 17 & 13 & 11 & 11 & 9 \\
\hhline{~~~~~~~~~~~~------}
\end{tabular}
\end{table}
\end{theorem}

\begin{theorem}
Values of $h_{nc}(4,k_1;3,k_2)$, $h_{isl}(4,k_1;3,k_2)$, $h(4,k_1;3,k_2)$, and $h_{nc}(5,k_1;3,k_2)$ are given in the tables below. Table rows correspond to the parameter $k_1$, and columns to $k_2$.
\begin{table}[h!]
\centering
\begin{tabular}{|c|cccc| c |c|cccc| c |c|ccccc| c |c|ccccc|}
\hhline{-----~-----~------~------}
\multicolumn{5}{|c|}{$h_{nc}(4,k_1;3,k_2)$} & & \multicolumn{5}{c|}{$h_{isl}(4,k_1;3,k_2)$} & & \multicolumn{6}{c|}{$h(4,k_1;3,k_2)$} & & \multicolumn{6}{|c|}{$h_{nc}(5,k_1;3,k_2)$} \\
\hhline{-----~-----~------~------}
& 0 & 1 & 2 & 3 & & & 0 & 1 & 2 & 3 & & & 0 & 1 & 2 & 3 & 4 & & & 0 & 1 & 2 & 3 & 4 \\
\hhline{-----~-----~------~------}
0 & 14 & 10 & 10 & 9 & & 0 & 17 & 11 & 11 & 9 & & 0 & 26 & 14 & 13 & 12 & 11 & & 0 & 20 & 13 & 12 & 12 & 11 \\
1 & 11 & 8 & 8 & 8 & & 1 & 12 & 9 & 8 & 8 & & 1 & 14 & 11 & 9 & 9 & 9 & & 1 & 15 & 11 & 10 & 10 & 10 \\
2 & 9 & 7 & 7 & 6 & & 2 & 10 & 7 & 7 & 6 & & 2 & 11 & 9 & 8 & 8 & 7 & & 2 & 11 & 9 & 8 & 8 & 7 \\
\hhline{-----~-----~------~------}
\end{tabular}
\end{table}
\end{theorem}

\begin{theorem}
The values of the function $h(3,k_1;3,k_2;3,k_3)$ are presented in the tables. The index $k_1$ determines the block number (from left to right, $k_1=0, 1, 2, 3, 4$), the rows correspond to $k_2$, and the columns correspond to $k_3$.
\begin{table}[h!]
\centering
\footnotesize
\begin{tabular}[c]{|r|*{9}{c}|c|r|*{4}c|c|r|*{3}c|c|r|*{2}c|c|r|*{1}c|}
\hhline{----------}
& 0 & 1 & 2 & 3 & 4 & 5 & 6 & 7 & 8 \\
\hhline{----------~-----}
0 & $-$& 33-& 26-& 23 & 21 & 20 & 20 & 20 & 19   & &   & 1 & 2 & 3 & 4 \\
\hhline{~~~~~~~~~~~-----~----}
1 & 33-& 19 & 17 & 16 & 15 & 15 & 15 & 14 & 14  & & 1 & 13 & 12 & 12 & 11 & &    & 2 & 3 & 4 \\
\hhline{~~~~~~~~~~~~~~~~~----~---}
2 & 26-& 17 & 15 & 14 & 14 & 14 & 13 & 13 & 13  & & 2 & 12 & 11 & 11 & 10 & & 2 & 10 & 9 & 9 & &   & 3 & 4 \\
\hhline{~~~~~~~~~~~~~~~~~~~~~~---~--}
3 & 23 & 16 & 14 & 13 & 13 & 13 & 12 & 12 & 12 & & 3 & 12 & 11 & 10 & 10 & & 3 & 9 & 9 & 9  & & 3 & 8 & 8 & & & 4 \\
\hhline{~~~~~~~~~~~~~~~~~~~~~~~~~~--}
4 & 21 & 15 & 14 & 13 & 12 & 12 & 12 & 12 & 12 & & 4 & 11 & 10 & 10 & 9  & & 4 & 9 & 9 & 8  & & 4 & 8 & 8 & & 4 & 7 \\
\hhline{~~~~~~~~~~~-----~----~---~--}
5 & 20 & 15 & 14 & 13 & 12 & 12 & 12 & 12 & 12 \\
6 & 20 & 15 & 13 & 12 & 12 & 12 & 11 & 11 & 11 \\
7 & 20 & 14 & 13 & 12 & 12 & 12 & 11 & 11 & 11 \\
8 & 19 & 14 & 13 & 12 & 12 & 12 & 11 & 11 & 11 \\
\hhline{----------}
\end{tabular}
\end{table}
\end{theorem}

\begin{theorem}
The table below presents additional values for multicolored homogeneous configurations. A dash indicates the non-existence of the required value (follows from \cite{chrom, cravioto}).
\begin{table}[h!]
\centering
\begin{tabular}{|c|ccccccccc|}
\hline
$k$ & 0 & 1 & 2 & 3 & 4 & 5 & 6 & 7 & 8 \\
\hline
$h(3,k;3,k;3,k;3,k)$     & -- & 29-& 18     & 13  & 12 & 10 & 9 & 9 & 9 \\
$h_{nc}(4,k;4,k;4,k)$    & -- & -- & 21-22 & 16  & 14 & 12 & 10 & 10 & 10 \\
$h_{isl}(4,k;4,k;4,k)$   & -- & -- & 24-    & 19- & 16 & 12 & 10 & 10 & 10 \\
$h(3,k;3,k;3,k;3,k;3,k)$ & -- & -- & 27-    & 22- & 18 & 16 & 14 & 12 & 11 \\
\hline
\end{tabular}
\end{table}
\end{theorem}

\section{Structure of the Paper}

The paper is organized as follows. Section~\ref{sec:defs} introduces basic definitions, formalizes the concept of a signotope, and describes the connection between the combinatorial properties of sets and their geometric realizability. Section~\ref{sec:model} is dedicated to describing the logical architecture of the research: it presents the CNF encoding of geometric predicates and algorithms for reducing Erdős--Szekeres type problems to the Boolean satisfiability problem (SAT) and Satisfiability Modulo Theories (SMT).

In Section~\ref{sec:linear}, the developed linear subreduction method is described in detail, which allowed for significant optimization of the search for geometric realizations in computationally intensive cases. Sections~\ref{sec:results_nc} and~\ref{sec:results_h} provide proofs of the main theorems formulated in the introduction, including the verification of new lower bounds. Finally, in Section~\ref{sec:concl}, open questions and perspectives for applying the proposed approach to related problems in combinatorial geometry are discussed.

\section{Combinatorial Model and SAT Formulation}
\label{sec:defs}
\subsection{Basic Definitions}

In this paper, we consider configurations of points in the plane in general position. To discretize the problem, we use the framework of order types and signotopes.

\begin{definition}[Order Type]
The order type of a finite set of points $\mathcal{X} = \{p_1, \dots, p_n\} \subset \mathbb{R}^2$ is a mapping $\chi \colon \binom{[n]}{3} \to \{-1, 1\}$ that assigns to each ordered triple of indices $(i, j, k)$ the orientation of the triangle $p_i p_j p_k$:
\[ \chi(i, j, k) = \operatorname{sgn} \det \begin{pmatrix} 1 & x_i & y_i \\ 1 & x_j & y_j \\ 1 & x_k & y_k \end{pmatrix}. \]
\end{definition}

\begin{definition}[Monotone Signotope]
A monotone signotope of rank 3 on $n$ elements is a mapping $\sigma \colon \binom{[n]}{3} \to \{-1, 1\}$ such that for any indices $1 \le i < j < k < l \le n$, the sequence of signs
\[ (\sigma(i, j, k), \sigma(i, j, l), \sigma(i, k, l), \sigma(j, k, l)) \]
has exactly one sign change.
\end{definition}

The connection between abstract combinatorial structures and Euclidean geometry is established through the concept of realization.

\begin{definition}[Geometric Realization]
A geometric realization of a signotope $\sigma$ is a set of points $\mathcal{X} = \{p_1, \dots, p_n\} \subset \mathbb{R}^2$ such that its order type $\chi$ coincides with $\sigma$, i.e., $\chi(i, j, k) = \sigma(i, j, k)$ for all $1 \le i < j < k \le n$. A signotope is called \textit{realizable} if there exists at least one geometric realization for it.
\end{definition}

\subsection{Motivation and the function $\tilde{h}(\cdots)$}

Calculating the classical function $h_*(n_1, k_1; \dots; n_c, k_c)$ in Euclidean space is difficult because the set of all point sets to be checked has the cardinality of the continuum. However, key geometric properties (convexity, a point lying inside a polygon) are invariant with respect to the order type. This allows mapping the problem onto the domain of abstract signotopes, the number of which for a fixed $n$ is large but finite.

The use of rank-3 monotone signotopes allows the geometric problem to be translated into the language of Boolean logic. This provides two key advantages:
\begin{enumerate}
\item \textbf{Discretization:} The infinite coordinate space is replaced by a finite space of Boolean variables $L_{abc} \in \{\text{true, false}\}$.
\item \textbf{Efficiency:} Modern SAT solvers use powerful heuristics to prune classes of configurations that certainly cannot contain the sought substructures.
\end{enumerate}

Hereafter in the text, the term {\it signotope} implies {\it monotone signotope of rank 3} unless explicitly stated otherwise.

By analogy with the classical case, let $\tilde{h}(n_1, k_1; \dots)$, $\tilde{h}_{nc}(n_1, k_1; \dots)$, and $\tilde{h}_{isl}(n_1, k_1; \dots)$ be the minimum integer $N$ such that any $c$-colored signotope on $N$ elements contains the corresponding monochromatic $n_i$-configuration (convex, non-convex, or island) with at most $k_i$ interior points.

Note that for the case $n_i=3$, the concepts of convex/arbitrary polygon and island are equivalent; therefore, we will use the general term {\it triangle} for them in the following text.

\begin{definition}[Maximal Signotope]
A signotope on $(N-1)$ elements is called \textit{maximal} if it does not contain any of the mentioned monochromatic $n_i$-polygons with the given constraint on the number of interior points, where $N = \tilde{h}_*(\cdots)$.
\end{definition}

The relationship between geometric and combinatorial functions is expressed by the fundamental inequality:
\[ h_*(n_1, k_1; \dots; n_c, k_c) \le \tilde{h}_*(n_1, k_1; \dots; n_c, k_c). \]

If it is possible to construct a configuration of $(N-1)$ points that is a geometric realization of at least one maximal signotope, then the equality $h_*(\cdots) = \tilde{h}_*(\cdots)$ is achieved. In the course of our research, such realizations were successfully found for almost all cases considered.

It is known that for $N \le 8$ points, all signotopes are realizable, which guarantees that the functions coincide. Starting from $N = 9$, non-realizable signotopes emerge. However, within the framework of the problem under consideration, as a rule, there exists not just one, but a whole family of different maximal signotopes. The value $h_*(\cdots)$ will be strictly less than $\tilde{h}_*(\cdots)$ only if all signotopes from this family simultaneously turn out to be non-realizable. Given that checking the realizability of an arbitrary signotope is an NP-hard problem, such examples have not yet been discovered in the literature. In the general case, $\tilde{h}_*(\cdots)$ provides an upper bound.

\vskip+0.2cm
\noindent\textbf{Open Question.} Does the equality $h_*(n_1, k_1; \dots; n_c, k_c) = \tilde{h}_*(n_1, k_1; \dots; n_c, k_c)$ always hold for arbitrary sets of parameters $n_i$ and $k_i$?

\section{Logical Architecture and CNF Encoding}
\label{sec:model}

Since in our model the points are initially ordered by the $x$-coordinate ($x_0 < x_1 < \dots < x_{N-1}$), the geometric conditions on the orientation of triples (variables $L_{abc}$) and point inclusion are substantially simplified. We associate a Boolean formula $\mathcal{F}_N$ in conjunctive normal form (CNF) with each configuration of $N$ points based on the following Boolean variables:

\begin{itemize}
\item $C_i(a) \in \{\text{true, false}\}$ --- point $a$ has color $i \in \{1, \dots, c\}$.
\item $L_{abc} \in \{\text{true, false}\}$ --- orientation of the triple $(a, b, c)$, where $\text{true} \iff \sigma(a, b, c) = +1$.
\item $EXT_{abc}(z) \in \{\text{true, false}\}$ --- point $z$ lies \textit{outside} the triangle $(a, b, c)$.
\item $TR_{abc}(q) \in \{\text{true, false}\}$ --- the triangle $(a, b, c)$ contains no more than $q$ interior points.
\end{itemize}

\subsection{Coloring and Inverse Color Encoding}
For each point $a \in \mathcal{X}$, a set of Boolean variables $\{C_1(a), C_2(a), \dots, C_c(a)\}$ is introduced. Our implementation uses \textit{inverse logic} for color assignment: the value $C_i(a) = \text{false}$ means that point $a$ is colored with color $i$, and $C_i(a) = \text{true}$ means that the point is of a different color.

This approach significantly simplifies the CNF representation of conditions on monochromatic figures, minimizing the number of negation operations. The conditions for correct coloring take the form:

1. \textbf{Color Existence Condition:} a point must be colored in at least one color (i.e., at least one variable must be \text{false}):
\begin{equation}
\overline{C_1(a)} \lor \overline{C_2(a)} \lor \dots \lor \overline{C_c(a)}
\end{equation}

2. \textbf{Uniqueness Condition (for $c > 1$):} a point cannot have more than one color. This means that for any two distinct colors $i$ and $j$, both variables cannot be simultaneously \text{false}:
\begin{equation}
\forall_{i < j}: C_i(a) \lor C_j(a)
\end{equation}

\subsection{Geometric Axioms of the Signotope}
To ensure that the variables $L_{abc}$ correspond to a monotone signotope, for each quadruple of indices $a < b < c < d$, constraints are imposed that guarantee exactly one sign change in the sequence $(L_{abc}, L_{abd}, L_{acd}, L_{bcd})$. In terms of forbidden configurations (negation of a conjunction), one of such conditions looks like this:
\begin{equation}
\overline{ L_{abc} \land \overline{L_{acd}} \land L_{bcd} } \land \overline{ \overline{L_{abc}} \land L_{acd} \land \overline{L_{bcd}} }
\end{equation}
After expansion using De Morgan's laws, we obtain CNF clauses:
\begin{equation}
(\overline{L_{abc}} \lor L_{acd} \lor \overline{L_{bcd}}) \land (L_{abc} \lor \overline{L_{acd}} \lor L_{bcd})
\end{equation}
The complete system of four such conditions (8 clauses) fully defines the required structure.

\subsection{Exterior Point and Density Variables}
Point $z$ can lie inside $\triangle abc$ only under the condition $a < z < c$. We define the conditions under which a point is external ($EXT$). Without loss of generality, consider the case $a < b < z < c$. Point $z$ is external if the orientations of triangles $azc$ and $bzc$ coincide:
\begin{equation}
(L_{azc} \iff L_{bzc}) \implies EXT_{abc}(z)
\end{equation}
Expanding the equivalence:
\begin{equation}
((L_{azc} \land L_{bzc}) \implies EXT_{abc}(z)) \land ((\overline{L_{azc}} \land \overline{L_{bzc}}) \implies EXT_{abc}(z))
\end{equation}
In CNF format:
\begin{equation}
(\overline{L_{azc}} \lor \overline{L_{bzc}} \lor EXT_{abc}(z)) \land (L_{azc} \lor L_{bzc} \lor EXT_{abc}(z))
\end{equation}

Density variables $TR_{abc}(q)$ encode the statement: ''there are no more than $q$ points inside triangle $abc$``. Let $P_{abc} = \{z : a < z < c, z \ne b\}$ be the set of all potential points that, by virtue of their $x$-coordinates, could be inside triangle $abc$. The density condition is formulated as follows: if in the set $P_{abc}$ there exists a subset $Z$ of size $|P_{abc}| - q$ consisting exclusively of exterior points, then the total number of interior points does not exceed $q$.
\begin{equation}
\left( \bigvee_{\substack{Z \subset P_{abc} \\ |Z| = |P_{abc}| - q}} \left( \bigwedge_{z \in Z} EXT_{abc}(z) \right) \right) \implies TR_{abc}(q)
\end{equation}
Expanding the implication yields a system of clauses for each $TR_{abc}(q)$:
\begin{equation}
\bigwedge_{ \substack{Z \subset P_{abc} \\ |Z| = |P_{abc}| - q}} \left( \bigvee_{z \in Z} \overline{EXT_{abc}(z)} \lor TR_{abc}(q) \right)
\end{equation}

\subsection{Conditions on Monochromatic Figures in Inverse Color Logic}
For each color $i$ and corresponding limit $k_i$, we introduce a prohibition on the existence of a convex $n_i$-gon containing no more than $k_i$ points inside.

By using inverse color encoding, the condition that all vertices of a potential polygon $\{a, b, c\}$ or $\{a, b, c, d\}$ have the same color $i$ is written as a disjunction of variables without negation signs.

For the case $n_i=3$ and a triple of points $a < b < c$ of the same color $i$, the {\it forbidden} configuration (empty or almost empty triangle) takes the form:
\begin{equation}
\overline{\overline{C_i(a)} \land \overline{C_i(b)} \land \overline{C_i(c)} \land TR_{abc}(k_i) }
\end{equation}
In CNF representation:
\begin{equation}
C_i(a) \lor C_i(b) \lor C_i(c) \lor \overline{TR_{abc}(k_i)}
\end{equation}

For $n_i=4$ and a quadruple of points $a < b < c < d$ of the same color $i$, configurations forming a convex quadrilateral with a total number of interior points not exceeding $k_i$ are forbidden. Within the monotone signotope model, this condition breaks down into two geometric scenarios:

1. \textbf{Case 4-cup / 4-cap:} $L_{abc} \iff L_{bcd}$. For all admissible combinations $q_1 + q_2 = k_i$, the forbidden configuration is described as:
\begin{equation}
\overline{\overline{C_i(a)} \land \overline{C_i(b)} \land \overline{C_i(c)} \land \overline{C_i(d)} \land (L_{abc} \iff L_{bcd}) \land TR_{abc}(q_1) \land TR_{acd}(q_2) }
\end{equation}
In CNF format, this is represented by a pair of clauses (for positive and negative orientation, respectively):
\begin{equation}
C_i(a) \lor C_i(b) \lor C_i(c) \lor C_i(d) \lor \overline{L_{abc}} \lor \overline{L_{bcd}} \lor \overline{TR_{abc}(q_1)} \lor \overline{TR_{acd}(q_2)}
\end{equation}
\begin{equation}
C_i(a) \lor C_i(b) \lor C_i(c) \lor C_i(d) \lor {L_{abc}} \lor {L_{bcd}} \lor \overline{TR_{abc}(q_1)} \lor \overline{TR_{acd}(q_2)}
\end{equation}

2. \textbf{Case 3-cup + 3-cap:} $L_{abd} \iff \overline{L_{acd}}$. In CNF format, these conditions are written as follows:
\begin{equation}
C_i(a) \lor C_i(b) \lor C_i(c) \lor C_i(d) \lor {L_{abd}} \lor \overline{L_{acd}} \lor \overline{TR_{abc}(q_1)} \lor \overline{TR_{bcd}(q_2)}
\end{equation}
\begin{equation}
C_i(a) \lor C_i(b) \lor C_i(c) \lor C_i(d) \lor \overline{L_{abd}} \lor {L_{acd}} \lor \overline{TR_{abc}(q_1)} \lor \overline{TR_{bcd}(q_2)}
\end{equation}

\subsection{Generalization to the Non-convex Case ($h_{nc}$)}
To calculate the function $h_{nc}$, where the convexity condition is not mandatory, the algorithm generates an expanded set of clauses. In this case, a prohibition is imposed on \textit{any} subset of 4 points of color $i$ if the total number of points inside the union of triangles forming the triangulation of the given quadruple does not exceed $k_i$. This is implemented through an exhaustive search of orientation combinations $L$ and imposing corresponding constraints on the density variables $TR$.

\subsection{Generalization to 4-islands ($h_{isl}$)}
When searching for islands, the criterion for prohibition is also the total number of points inside the convex hull of a quadruple of points of color $i$. Depending on the relative positions, this condition is formulated as follows:
\begin{enumerate}
\item If the quadruple is in a convex position, the sum of interior points in the two triangles making up its triangulation is taken into account.
\item If one point is inside the triangle formed by the other three, the number of points inside this (outer) triangle is considered, and the limit of interior points is increased by one (accounting for the innermost point of the quadruple itself).
\end{enumerate}

\subsection{Symmetry Breaking}
To significantly optimize the search space, we fix the orientation of \textit{all} triples containing the point with the smallest index:
\begin{equation}
\forall_{0 < b < c < n}: \quad L_{0bc} = \text{true}
\end{equation}
According to Scheucher's results \cite{scheucher2020}, any finite set of points in the plane in general position can be renumbered such that the condition $L_{0bc} = \text{true}$ is satisfied for all $b < c$, and then affinely transformed into an order-type equivalent set whose points are ordered by the $x$-coordinate. Thus, this constraint does not narrow the search space and is applicable to all abstract monotone signotopes.

\section{Geometric Realization and Search for Extremal Configurations}
\label{sec:linear}

To confirm the equality $h_*(\dots) = \tilde{h}_*(\dots)$, it is necessary to demonstrate the existence of a point set $\mathcal{X}$ realizing at least one maximal signotope. In this work, we applied two complementary methods for finding coordinates.

\subsection{Local Stochastic Search Method}
This iterative approach is based on the direct variation of point coordinates in $\mathbb{R}^2$. The process begins with the generation of a random distribution of $N$ points in the plane. At each step, the algorithm calculates the number of {\it forbidden} structures (for example, monochromatic convex $n_i$-polygons with no more than $k_i$ interior points) to be excluded.

The algorithm selects a random point and moves it to a new, randomly chosen position. If the modified configuration is characterized by a smaller or equal number of forbidden figures compared to the previous state, the movement is fixed; otherwise, the point returns to its original position. The process is repeated until the undesirable structures are fully eliminated.

The \textit{advantage} of the method is the possibility of searching for high-symmetry realizations (by coordinated movement of groups of points while preserving the symmetry group). The \textit{disadvantage} is the low convergence rate in the vicinity of local minima. Note that the efficiency of the approach can be significantly improved using the simulated annealing algorithm, where the probability of accepting a {\it worse} step decreases exponentially with a drop in the temperature parameter, allowing the system to escape local optima.

\subsection{Linear Subreduction Method}
Synthesis of Boolean signotope constraints with linear arithmetic theory within the SMT framework proved to be the most effective method. The Z3 SMT solver from Microsoft \cite{z3} was used as the primary tool.

The general task of finding coordinates for a given signotope reduces to solving a system of non-linear inequalities. In our model, the connection between logic variables $L_{abc}$ and point coordinates $p_i = (x_i, y_i)$ is specified in the form of disjunctive conditions:
\begin{equation}
L_{abc} ~~ \lor ~~ (x_b - x_a)y_c + (x_a - x_c)y_b + (x_c - x_b)y_a \ge 1
\end{equation}
\begin{equation}
\overline{L_{abc}} ~~ \lor ~~ (x_b - x_a)y_c + (x_a - x_c)y_b + (x_c - x_b)y_a \le -1
\end{equation}

Since these conditions are quadratic relative to variables $x$ and $y$, SMT solvers effectively find solutions only for small values of $N$. We optimized the task by fixing the abscissae of the points (using a uniform distribution $x_i = i$ or an exponential spacing $x_i \in \{\dots, -C^2, -C, -1, [0], 1, C, C^2, \dots\}$, where the entry $[0]$ denotes the presence of a central point only for configurations with an odd $N$). This turns quadratic inequalities into \textbf{linear} ones relative to the ordinates $y_i$. With such a linearized problem, the SMT solver performs orders of magnitude faster.

The fundamental difference between our approach and classical verification methods (e.g., \cite{scheucher2020}) is that we do not look for a realization of a specific, pre-determined signotope. The \textbf{full logic formula} of the problem is fed to the SMT solver. Thus, the solver looks for coordinates for \textit{any} of the entire set of admissible signotopes, satisfying both combinatorial and geometric constraints simultaneously. This significantly expands the search area and allows for finding realizations even in cases where the space of admissible signotopes is extremely limited.

Note that for searching the ordinates $y_i$ in the SMT model, we utilized the integer data type (\texttt{Int}) instead of the real type (\texttt{Real}). From the perspective of the Z3 solver's architecture, this allows the use of specialized linear integer arithmetic (LIA) algorithms, which in some cases demonstrate better convergence on dense systems of inequalities.

Although point coordinates in the plane are a priori real numbers, transitioning to integers in our problem does not limit the generality of the search. Since the orientation function $L_{abc}$ is invariant under positive scaling, any real solution to the system of linear inequalities can be approximated by a rational one due to the openness of the feasible region, and then brought to integer form by multiplying by a common denominator. Thus, the use of the \texttt{Int} type not only makes the sought coordinates more representative but also optimizes the computational load on the SMT solver.

It is worth noting that fixing the abscissae $x_i$ imposes additional constraints on the geometry of the set $\mathcal{X}$, which theoretically could lead to the loss of some realizations. Strictly speaking, an equivalent point set with a given set of $x$-coordinates does not exist for every realizable signotope. Nevertheless, this approach is dictated by the need to achieve an acceptable computation speed.

The motivations for choosing fixed $x_i$ values are as follows:
\begin{enumerate}
\item \textbf{Exponential Spacing:} Using the grid $x_i \in \{\dots, -C^2, -C, -1, 0, 1, C, C^2, \dots\}$ allows for the imitation of {\it stretched} configurations, which are often found in extremal examples of order type theory.
\item \textbf{Density of Realizations:} Empirical evidence suggests the space of geometric realizations for most signotopes is sufficiently vast, allowing the sought configuration to be found even in a limited subspace with fixed $x_i$.
\item \textbf{Empirical Completeness:} In almost all cases we considered where the SAT solver confirmed the existence of an abstract signotope, the SMT solver successfully found its integer realization even with the simplest distribution $x_i = i$.
\end{enumerate}

Thus, the linear subreduction method is a powerful tool for fast verification, allowing the automated confirmation of the equality $h_*(\cdots) = \tilde{h}_*(\cdots)$ for the majority of configurations.

\section{The ES\_color.py Software Package}

To conduct numerical experiments and automatically generate logical formulas, a specialized software package was developed in the Python language --- \texttt{ES\_color.py}. The program supports generating output data in DIMACS (for SAT solvers) and SMT-LIB v2 (for SMT solvers) formats.

\subsection{Configuration Parameters}
The script is executed from the command line using the following key parameters:

\begin{itemize}
\item \texttt{n=N} --- total number of points $N$ in the sought configuration.
\item \texttt{tr\_i=k\_i} --- constraint on the maximum number of interior points for triangles of color $i$. For example, the parameters \texttt{tr1=0 tr2=0} initiate the prohibition of empty triangles of the first and second colors, respectively.
\item \texttt{cv\_i=k\_i} --- exclusion of configurations containing convex quadrilaterals of color $i$ considering the limit $k_i$ (problem $h$).
\item \texttt{nc\_i=k\_i} --- exclusion of any (not necessarily convex) quadrilaterals of color $i$ containing no more than $k_i$ points inside (problem $h_{nc}$).
\item \texttt{is\_i=k\_i} --- exclusion of monochromatic 4-islands (problem $h_{isl}$).
\item \texttt{sb=off} --- deactivation of symmetry breaking algorithms.
\item \texttt{xgrid=C} --- switching to the mode of generating an SMT2 formula to search for a geometric realization. If $C=1$, a linear grid of abscissae $x_i=i$ is used. If $C>1$, the base of an exponential grid is set: $x_i \in \{\dots,-C^2,-C, -1, [0], 1 ,C, C^2,\dots\}$. This mode automatically deactivates the \texttt{sb} option.
\end{itemize}

\subsection{Usage Examples and Integration with Solvers}

To search for an abstract signotope, the result is directed to a SAT solver. Below is an example of verifying the equality $\tilde{h}_{nc}(4,0;3,0)=14$:
\begin{lstlisting}[language=bash]
./ES_color.py nc1=0 tr2=0 n=13 | kissat
SATISFIABLE
./ES_color.py nc1=0 tr2=0 n=14 | kissat
UNSATISFIABLE
\end{lstlisting}

To confirm geometric realizability and calculate the ordinates $y_i$, the \texttt{xgrid} mode is used. The generated SMT2 formula is passed to the Z3 solver via the standard stream (pipe):
\begin{lstlisting}[language=bash]
./ES_color.py nc1=0 tr2=0 n=13 xgrid=1 | z3 -in | sed ':a;N;s/)\n (/) (/g;ba'
sat
((x0 0) (y0 (- 278)) (k1 true) (k14 false))
((x1 1) (y1 (- 172)) (k2 false) (k15 true))
...
((x11 11) (y11 186) (k12 false) (k25 true))
((x12 12) (y12 208) (k13 true) (k26 false))
\end{lstlisting}

The Z3 output in this case contains not only the verdict \texttt{sat/unsat} but also, in the case of \texttt{sat}, the specific values of the coordinates $y_i$ and the color distribution $C_i(a)$ for all points, allowing for instant verification and visualization of the found example.

\section{Answer Set Programming}
For independent verification of the results obtained using the imperative SAT generator in Python, an alternative logical model was developed in a declarative programming language. Using the \texttt{clingo} system \cite{clingo} allows the description of geometric constraints in terms of predicate logic, which ensures a high level of abstraction.

The fundamental advantage of this approach lies in the exceptional conciseness of the code: the entire logic of the problem --- including signotope axioms, point inclusion predicates, and conditions on forbidden configurations --- is implemented in a few dozen lines. Such transparency significantly simplifies the audit of the algorithm and minimizes the likelihood of introducing implementation errors that may occur during the manual formation of complex CNF structures.

Below is the full text of the ASP model used for cross-checking the calculations:

\begin{lstlisting}[ language=clingo]
pt(0..n-1).

%  --- SELECT TARGET POLYGONS ---
%  Format: ins(Type, ColorID, MaxInteriorPoints)
ins(pr,p1,pr1;
    tr,t1,tr1; tr,t2,tr2; tr,t3,tr3; tr,t4,tr4; tr,t5,tr5;
    cv,c1,cv1; cv,c2,cv2; cv,c3,cv3;
    cv,i1,is1; cv,i2,is2; cv,i3,is3;
    is,i1,is1+1; is,i2,is2+1; is,i3,is3+1;
    (nc;cv),n1,nc1; (nc;cv),n2,nc2; (nc;cv),n3,nc3).
#const is1=-2. #const is2=-2. #const is3=-2.

%  --- GENERATORS ---
%  Assign exactly one color to each point
1{c(A,Z): ins(_,Z,I),I=0..99}1 :- pt(A).
#heuristic c(A,Z): pt(A), ins(_,Z,I), I=0..99. [10, level]
%  Assign exactly one rotation to each triplet (Chirotope base)
1{l(A,B,C,R): R=(-1;1)}1 :- pt(A),pt(B),pt(C), A<B,B<C.

%  --- GEOMETRIC CONSTRAINTS (Axioms for Signotope) ---
%  These ensure that the relative positions of points are physically possible
:- l(A,B,C,R), l(A,C,D,-R), l(B,C,D,R).
:- l(A,B,C,R), l(A,B,D,-R), l(A,C,D,R).
:- l(A,B,C,R), l(A,B,D,-R), l(B,C,D,R).
:- l(A,B,D,R), l(A,C,D,-R), l(B,C,D,R).
%  Symmetry breaking: fix orientation of the first triplets to avoid rotated solutions
:- l(0,B,C,-1), sb!=off.

%  --- INTERIOR POINTS LOGIC ---
%  Define if point X is inside triangle (A,B,C) based on relative orientations
i(A,B,C,X) :- l(A,X,B,R), l(A,X,C,-R), B<C.
i(A,B,C,X) :- l(B,X,C,R), l(A,X,C,-R), A<B.

%  Calculate if triangle (A,B,C) contains no more than J interior points
tr(A,B,C,J) :- pt(A),pt(B),pt(C), A<B,B<C, {i(A,B,C,X)}<=J, J=0..I, ins(_,_,I), I=0..99.

%  --- SHAPE INTEGRITY CONSTRAINTS ---

%  Pairs
:- ins(pr,Z,_),  c(A,Z),c(B,Z), A<B.

%  Triangles
:- ins(tr,Z,I),  c(A,Z),c(B,Z),c(C,Z),  A<B,B<C,  tr(A,B,C,I).

%  Convex quadrilaterals
:- ins(cv,Z,I),  l(A,B,C,R), l(B,C,D,R),   c(A,Z),c(B,Z),c(C,Z),c(D,Z),  A<B,B<C,C<D,  tr(A,B,C,I1), tr(A,C,D,I2), I1+I2=I.
:- ins(cv,Z,I),  l(A,B,D,R), l(A,C,D,-R),  c(A,Z),c(B,Z),c(C,Z),c(D,Z),  A<B,B<C,C<D,  tr(A,B,C,I1), tr(B,C,D,I2), I1+I2=I.

%  Non-convex quadrilaterals
:- ins(nc,Z,I),  l(A,B,D,R), l(A,B,C,-R),  c(A,Z),c(B,Z),c(C,Z),c(D,Z),  A<B,B<C,C<D,  tr(A,B,C,I1), tr(A,B,D,I2), I1+I2=I.
:- ins(nc,Z,I),  l(A,B,D,R), l(A,B,C,-R),  c(A,Z),c(B,Z),c(C,Z),c(D,Z),  A<B,B<C,C<D,  tr(A,B,C,I1), tr(B,C,D,I2), I1+I2=I.
:- ins(nc,Z,I),  l(A,B,D,R), l(A,B,C,-R),  c(A,Z),c(B,Z),c(C,Z),c(D,Z),  A<B,B<C,C<D,  tr(A,B,D,I1), tr(B,C,D,I2), I1+I2=I.
:- ins(nc,Z,I),  l(A,C,D,R), l(B,C,D,-R),  c(A,Z),c(B,Z),c(C,Z),c(D,Z),  A<B,B<C,C<D,  tr(A,B,C,I1), tr(A,C,D,I2), I1+I2=I.
:- ins(nc,Z,I),  l(A,C,D,R), l(B,C,D,-R),  c(A,Z),c(B,Z),c(C,Z),c(D,Z),  A<B,B<C,C<D,  tr(A,B,C,I1), tr(B,C,D,I2), I1+I2=I.
:- ins(nc,Z,I),  l(A,C,D,R), l(B,C,D,-R),  c(A,Z),c(B,Z),c(C,Z),c(D,Z),  A<B,B<C,C<D,  tr(A,C,D,I1), tr(B,C,D,I2), I1+I2=I.

%  4-islands
:- ins(is,Z,I), l(A,B,D,R), l(A,B,C,-R),   c(A,Z),c(B,Z),c(C,Z),c(D,Z),  A<B,B<C,C<D,  tr(A,C,D,I).
:- ins(is,Z,I), l(A,C,D,R), l(B,C,D,-R),   c(A,Z),c(B,Z),c(C,Z),c(D,Z),  A<B,B<C,C<D,  tr(A,B,D,I).

#show c/2.
#show l/4.
\end{lstlisting}

\subsection{Dynamic Parameterization and Constraint Activation Mechanism}
A feature of the presented ASP model is the use of the \texttt{ins(Type, ColorID, MaxPoints)} predicate as a system registry of active constraints. This allows for the dynamic determination of the number of colors involved and the types of geometric prohibitions without modifying the program's source code.

\begin{enumerate}
\item \textbf{Parameter Activation Mechanism:}
The key logic of the model is embedded in the rule for generating the coloring:
\begin{lstlisting}[language=clingo, numbers=none]
1{c(A,Z) : ins(_,Z,I), I=0..99}1 :- pt(A).
\end{lstlisting}
This rule uses the matching method with the interval \texttt{I=0..99}. When a parameter is set via the command line (for example, \texttt{-c nc1=0}), the constant \texttt{nc1} receives a specific integer value. The condition \texttt{ins(nc, n1, I), I=0..99} becomes true, as the value of \texttt{I} falls within the specified range, which {\it activates} the color identifier \texttt{n1} for the generator \texttt{c(A,Z)}.

If a parameter for a specific color (for example, \texttt{tr5}) is not passed, the corresponding constant remains undefined, and the \texttt{ins} predicate for that color is not evaluated. As a result, the generator \texttt{c(A,Z)} excludes this color from the search space when coloring points. Thus, the number of active colors in the model strictly corresponds to the number of given external constraints.
\end{enumerate}

\begin{enumerate}
\setcounter{enumi}{1}
\item \textbf{Constraint Selectivity:}
Each exclusion rule (lines 29--50 of the listing) is associated with the corresponding type from the \texttt{ins} predicate. This ensures that the logic, for example, of non-convex quadrilaterals (\texttt{nc}), is applied exclusively to those points whose color \texttt{Z} was activated through the corresponding parameter. This approach allows combining diverse geometric requirements for different color classes within a single program.

\item \textbf{Exhaustive Search Optimization:}
For Erdős--Szekeres type problems, the proof of the absence of solutions (\texttt{UNSAT} verdict) is critical, as it serves as strict confirmation of the exact value of the function $\tilde{h}_*$. For complex configurations, the verification procedure may require significant computational resources.

The \texttt{\#heuristic} directive instructs the solver to prioritize decisions regarding the coloring variables \texttt{c(A,Z)}. Setting a high priority (\texttt{10, level}) forces the system to distribute points by color first, which significantly accelerates the Unit Propagation procedure and allows for more efficient pruning of search tree branches that do not contain valid configurations. This is a decisive factor for obtaining the \texttt{UNSAT} result in an acceptable time.

Additionally, to speed up the verification process, specific settings for the \texttt{clingo} solver are used:
\begin{itemize}
\item \texttt{--configuration=frumpy} --- activation of a mode oriented towards intensive depth-first search with frequent conflicts. This preset demonstrated the highest efficiency when analyzing problems with a small number of colors (1--2), where the constraint density is high and the search tree requires rigid and fast pruning of non-promising branches.
\item \texttt{--configuration=crafty} --- a preset that proved to be most effective for problems with three or more colors. Its advantage in multicolored configurations is due to a more flexible restart mechanism and conflict-learning heuristics, which allow for faster structure finding in a more sparse and multidimensional search space.
\item \texttt{--sat-p=3} --- use of a SAT-level preprocessor to simplify the logical formula directly during the search process (on the fly).
\item \texttt{--heuristic=Domain} --- prioritization of variable selection based on their domain structure (in combination with the author's \texttt{\#heuristic} directive), which is critically important for the prompt pruning of unsatisfiable search tree branches.
\end{itemize}
\end{enumerate}

Example of a run to verify the value $\tilde{h}_{nc}(4,0; 3,0) = 14$ for bicolored sets:
\begin{lstlisting}[language=bash]
clingo ES_color.lp --configuration=frumpy --sat-p=3 --heuristic=Domain -c nc1=0 -c tr2=0 -c n=13
SATISFIABLE
clingo ES_color.lp --configuration=frumpy --sat-p=3 --heuristic=Domain -c nc1=0 -c tr2=0 -c n=14
UNSATISFIABLE
\end{lstlisting}
In this scenario, the solver activates two colors (\texttt{n1} and \texttt{t2}), applies the corresponding types of geometric constraints to them, and performs an exhaustive search to formally prove the absence of solutions.

\subsection{Finding Coordinates Using the clingo-lpx Solver}
As an alternative method for finding geometric coordinates, the \texttt{clingo-lpx} solver was investigated. This extension integrates the simplex method directly into the Answer Set Programming (ASP) search process, allowing the formulation of linear constraints over rational numbers within the logical program.

To implement the linear subreduction method, an additional module was developed, formalizing the connection between signotope orientations and point ordinates $y_i$ with fixed abscissae $x_i$:

\begin{lstlisting}[language=clingo, caption={lpx --- extension for linking logic variables with coordinates}]
#const sb=off.
#const xgrid=1.

x(N,N) :- pt(N), xgrid=1.

x(0,n/2) :- xgrid>1, n\2=1.
x(-(xgrid**(n/2-N-1)),N) :- pt(N), N<n/2, xgrid>1.
x( xgrid**(N-(n-1)/2-1), N) :- pt(N), N>(n-1)/2, xgrid>1.

&sum {KA*y(XA); KB*y(XB); KC*y(XC)} >= 1 :- l(A,B,C,1),  x(XA,A),x(XB,B),x(XC,C), KA=XC-XB, KB=XA-XC, KC=XB-XA, xgrid>0.
&sum {KA*y(XA); KB*y(XB); KC*y(XC)} <= -1:- l(A,B,C,-1), x(XA,A),x(XB,B),x(XC,C), KA=XC-XB, KB=XA-XC, KC=XB-XA, xgrid>0.
\end{lstlisting}

\subsubsection{Features and Performance}
The \texttt{clingo-lpx} solver allows for finding exact rational values of the ordinates $y_i$ (in the form of common fractions) that satisfy the given signotope. Below is an example of the work for the problem $h_{nc}(4,0; 3,0) = 14$ at $N=13$:
\begin{lstlisting}
clingo-lpx lpx ES_color.lp -c nc1=0 -c tr2=0 -c n=13 -c xgrid=1
...
y(0)=4719/224 y(1)=11 y(2)=7801/448 y(3)=-323/112 y(4)=-10833/448 y(5)=-10411/224
y(6)=4143/224 y(7)=-3489/112 y(8)=-1231/32 y(9)=4591/896 y(10)=1 y(11)=0 y(12)=0
SATISFIABLE
\end{lstlisting}

Despite the deep integration of arithmetic and logic, empirical tests have shown that the \texttt{clingo-lpx} solver almost always lags behind the Z3 SMT solver in terms of performance. This is likely due to the fact that Z3 implements more aggressive conflict-driven clause learning (CDCL) strategies for mixed (logic-arithmetic) problems.

Nevertheless, the use of \texttt{clingo-lpx} remains promising for configurations with small $N$, as it allows the full model to be described within a single declarative language without the need for external CNF generators and intermediate data formats. This makes the \texttt{clingo-lpx}-based approach an effective and convenient extension of the main ASP code for the rapid verification of new geometric hypotheses.

\section{Computational Experiment Results}
\label{sec:results_nc}

\subsection{Methodology and Tools}
To verify combinatorial configurations in this study, SAT solvers representing various stages of the development of logical inference search algorithms were used. \textit{MiniSat} \cite{minisat} and \textit{Glucose} \cite{glucose} were chosen as baselines. These classical representatives of the CDCL architecture have been the industry standard for the past decades due to their stability and predictability when working with high-density combinatorial problems.

As a modern high-performance solution (State-of-the-Art), the \textit{Kissat} solver \cite{kissat} (version 4.0.4) was used, which demonstrated the best results in recent international SAT Competitions. Unlike universal SMT systems such as \textit{Z3}, which are optimized to support a wide range of logical theories and carry additional computational overhead for their processing, \textit{Kissat} is extremely specialized exclusively in Boolean satisfiability problems. The high efficiency of this tool is achieved through aggressive inprocessing formula simplification methods and the dynamic adaptation of heuristics to the structure of a specific instance.

\subsection{Computationally Simple Cases}
This category includes all scenarios for which the value of the function $\tilde{h}_*(\dots)$ does not exceed 20, with the exception of three specific cases:
\[ \tilde{h}(3,2; 3,2; 3,2; 3,2) = 17, \quad \tilde{h}(3,4; 3,4; 3,4; 3,4; 3,4) = 17, \quad \tilde{h}_{isl}(4,3; 4,3; 4,3) = 19, \]
a detailed analysis of which is presented in the next section.

The running time of SAT/ASP solvers for all configurations with up to $N = 14$ points were below 12 seconds. As the set size increased to $N = 17$, the verification time generally did not exceed 1000 seconds. The only exception was the case $\tilde{h}_{nc}(3,3;3,3;3,3)=16$, which required approximately 52,000 CPU seconds. For problems with $N = 20$ points, the search time was less than 70,000 seconds.

A comparative analysis of the performance of various tools is available in the project repository\footnote{\url{https://github.com/koshelevv/Erdos-Szekeres/tree/main/colored_points}}; \textit{Kissat} proved to be the most effective solver for these problems in practice.

For all specified cases, the linear subreduction method allowed the construction of extremal point configurations on a linear abscissa grid ($x_i = i$). In most cases, the coordinate search time ranged from a few seconds to 750 seconds. 

For three computationally intensive tasks:
\begin{equation}
h_{isl}(4,0;4,2)=16, \quad h(4,1;4,2)=16, \quad h(4,0;4,4)=17
\end{equation}
the time required to generate the realization was 4, 2.6, and 3 hours, respectively.

\subsection{Computationally Complex Cases}
This section presents the results for the most resource-intensive configurations. Summary data on solver runtimes and parameters of the found realizations are provided in the table below.

\begin{table}[ht!]
\centering
\small
\begin{tabular}{|l|c|cccc||ccc|}
\multicolumn{9}{c}{\footnotesize Computational results for complex cases; verification time is given in CPUh.} \\
\hline
& & & \multicolumn{2}{c}{\textbf{clingo --heuristic}} & & \multicolumn{3}{c|}{\textbf{\scriptsize Linear Subreduction Method}} \\

\textbf{Parameters} & $\mathbf{\tilde{h}(\cdots)}$ & \textbf{kissat} & \textbf{Default} & \textbf{Domain} & \textbf{Decomp.} & $\mathbf{\tilde{h}-1}$ & $\mathbf{|\mathcal{X}|}$ & $\mathbf{t_{\mathcal{X}}}$ \\ \hline \hline
$h(3,2; 3,2; 3,2; 3,2)$ & 18 & $\circ$ & 447 & 294 & $37^G$ & 17 &$17^{(4)}$ & 8.8 \\
$h(3,4; 3,4; 3,4; 3,4; 3,4)$ & 18 & 29      & $\circ$ & 324 & $\circ$ & 17 &$17^{(4)}$ & 4.5 \\
$h(3,0; 3,0; 3,4)$           & 21 & 43      & 67 & 73 & $\circ$ & 20 & 20 & 4.2 \\
$h(3,0; 3,0; 3,3)$           & 23 & $\circ$ & 884 & 1067 & $\circ$ & 22 & 22 & 5.8 \\
$h(4,0; 4,3)$                & 21 & 5.6     & 8.7 & 8.7 & $\circ$ & 20 & 19 & 17 \\
$h_{isl}(4,0; 4,1)$          & 22 & 57      & 55 & 100 & $\circ$ & 21 & 21 & 33 \\
$h(4,0; 4,2)$                & 25 & $\circ$ & 1619 & $\circ$ & $563^G$ & 24 &$22^{(4)}$ & 1783 \\
$h(4,0;3,0)$                 & 26 & $\circ$ & 1193 & $\circ$ & $353^G$ & 25 & 25 & LSS \\
$h_{nc}(4,0; 4,0)$           & 26 & $\circ$ & 1583 & $\circ$ & $350^G$ & 25 & 25 & 97 \\
$h_{nc}(4,2; 4,2; 4,2)$      & 22 & $\circ$ & $\circ$ & $\circ$ & $880^M$ & 21 & 20 & 543 \\ \hline
$h_{isl}(4,3; 4,3; 4,3)$     & \multicolumn{5}{c||}{} & $18^*$ & 18 & 5 \\
$h_{isl}(4,2; 4,2; 4,2)$     & \multicolumn{5}{c||}{} & $27^*$ & $24^{(4)}$ & 61 \\
$h(3,0; 3,0; 3,2)$           & \multicolumn{5}{c||}{Experiments in progress} & $25^*$ & 25 & 1 \\
$h(3,0; 3,0; 3,1)$           & \multicolumn{5}{c||}{(decomposition and significant resources required)} & $34^*$ & 32 & 855 \\
$h(4,0; 4,1)$                & \multicolumn{5}{c||}{} & $33^*$ & 28 & 756 \\
$h(3,1; 3,1; 3,1; 3,1)$      & \multicolumn{5}{c||}{} & $35^*$ & 28 & LSS \\
$h(3,3; 3,3; 3,3; 3,3; 3,3)$ & \multicolumn{5}{c||}{} & $22^*$ &$21^{(4)}$ & 8302 \\
$h(3,2; 3,2; 3,2; 3,2; 3,2)$ & \multicolumn{5}{c||}{} & $32^*$ & 26 & 821 \\ \hline
$h_{isl}(4,0; 4,0)$ & \multicolumn{5}{c||}{\text{Existence of $\tilde{h}$ and $h$ is open;}} & $48^{**}$ & 35 & manual \\
$h(4,0; 4,0)$ & \multicolumn{5}{c||}{\text{current lower bounds are presented.}} & $90^{**}$ & 46 & manual \\ \hline
\end{tabular}

\scriptsize
\vskip 0.1cm
$\circ$ --- experiment in progress or not conducted due to data sufficiency; \\
$^*$ --- size of the largest signotope found (we assume these are maximal); \\
$^{**}$ --- size of the largest signotope found without a guarantee of maximality; \\
$^G$ / $^M$ --- use of Glucose / Minisat solvers on subformulas after decomposition; \\
$^{(4)}$ --- use of abscissa grid No. 4 instead of the default grid No. 1; \\
$|\mathcal{X}|$ --- cardinality (number of points) of the found geometric realization; \\
$t_{\mathcal{X}}$ --- coordinate search time using the linear subreduction method; \\
LSS --- Local Stochastic Search method; \\
manual --- solution found manually, computer used only for visualization.
\end{table}

To verify the most labor-intensive cases, we applied a method of decomposing the original logical formula into subproblems by fixing the coloring of the initial segment of points. In several cases (particularly for homogeneous multicolored problems), we utilized the {\it color equivalence} property, which allows, without loss of generality, fixing the color of the first point or specific combinations of initial colorings to reduce the number of subproblems:

\begin{itemize}
    \item For $\tilde{h}(3,2; 3,2; 3,2; 3,2) = 18$, the symmetry of the four colors allowed us to consider only five variants of fixing the colors of the first three points: $123, 122, 121, 112$, and $111$, instead of the theoretically possible $4^3 = 64$ initial colorings. The case $111$ was immediately excluded due to the obvious presence of a monochromatic empty triangle. The \textit{Glucose} solver (2023) completed the calculations in 24, 0.5, 7, and 5 CPUh, respectively.

    \item The value $\tilde{h}(4,0; 4,2) = 25$ was obtained by decomposition into 2981 subformulas with fixed colors for the first 12 points. Trivial sets (e.g., 5 consecutive points of the same color) were excluded from the search. The solution time for the subproblems varied from fractions of a second to 41,240 seconds, with total costs amounting to 563 CPUh.

    \item For $\tilde{h}(4, 0; 3, 0) = 26$, the problem was divided into 2387 subformulas (fixing colors for the first 13 points). Configurations with five consecutive points of the first color or three consecutive points of the second color were excluded. With a maximum time per subproblem of 34,600 seconds, the total costs for the \textit{Glucose} solver were 353 CPUh.

    \item The proof of $\tilde{h}_{nc}(4,0;4,0) = 26$ required decomposition into 1706 subproblems (fixing colors for the first 13 points). Due to color symmetry, the first point was always assigned color 1. Sets with four consecutive points of the same color were excluded. The final time (\textit{Glucose}) was 350 CPUh, with a maximum per subproblem of 48,600 seconds.

    \item Verification of $\tilde{h}_{nc}(4,2; 4,2; 4,2) = 22$ was performed through 23,907 subproblems (fixing colors for the first 12 points). Using symmetry for the first point and filtering trivial sets (6 consecutive points, 4 of which are monochromatic) allowed for optimization of the search. The \textit{MiniSat} solver utilized 880 CPUh (maximum 13,670 seconds).
\end{itemize}

For the critical cases mentioned above, the \textit{Glucose} and \textit{MiniSat} solvers do not reach an UNSAT result on the original formula without decomposition, even when significantly exceeding the total time limit. Modern latest-generation SAT solvers, such as \textit{CaDiCaL} and \textit{Kissat}, which use advanced inprocessing methods, are potentially capable of processing such problems without prior partitioning. Tests on lower-power instances confirm their high efficiency; however, the full verification of the most complex formulas in a single pass remains a subject of current research.

The application of the ASP approach enabled the direct verification of several labor-intensive cases. Unlike the \textit{Glucose} and \textit{MiniSat} solvers, which required partitioning the problem into thousands of subproblems, the \textit{clingo} system successfully handled the analysis of the original formulas in their entirety. Although the total runtime of \textit{clingo} in some scenarios exceeded the cumulative CPU time of parallel runs, the ability to verify without manual intervention in the search structure significantly increases the reliability and transparency of the experiment.

A comparison of the ultimate computational capabilities demonstrates a significant asymmetry: while for two-color problems modern solvers allow the verification of configurations up to $N=26$ points, the transition to four or more colors substantially constrains the scope of accessible computations to $N=18$. Such a sharp decrease in the threshold confirms that multi-coloring is a critical factor exponentially influencing the complexity of geometric models of this type.

Our experience confirms that the choice of the optimal preset in \textit{clingo} directly correlates with the number of colors: \texttt{frumpy} is most effective for 1–2 colors, while for 3 or more colors, \texttt{crafty} maintains a stable advantage.

Part of the results presented in the table was obtained at a stage before the implementation of the \texttt{--heuristic=Domain} optimization. Nevertheless, the achieved computational correctness allowed for the fixing of the sought values. Ongoing recalculations with the \texttt{Domain} parameter are aimed at refining the performance boundaries of the method. The only case that retained its exceptional laboriousness and required decomposition even within the ASP approach remains the configuration $\tilde{h}_{nc}(4,2; 4,2; 4,2)$.

\subsubsection{Search for Coordinates: Linear Subreduction Method and Alternative Approaches}

For the case $h_{nc}(4,0; 4,0)$, a geometric realization of 25 points was constructed using the linear subreduction method:
\coordinates{(0,$-$746538137,B), (1,$-$3046660,A), (2,$-$646999721,A), (3,3839774366,A), (4,3093155294,B), (5,7276998680,B), (6,5993695355,A), (7,8,B), (8,106660406,A), (9,$-$1386562914,B), (10,$-$2879830757,A), (11,18199678130,A), (12,16147587405,B), (13,639962389,B), (14,$-$2228740740,A), (15,$-$8,A), (16,12743409974,A), (17,11892365617,B), (18,12428608367,A), (19,24846857556,B), (20,27927600703,A), (21,35677096872,A), (22,38225025694,B), (23,$-$15821880854,B), (24,$-$16319040999,B).}

For the function $h(4,0; 3,0)$, a configuration of 25 points exhibiting 3-fold symmetry was found using the local stochastic search method; the linear subreduction method was not applied in this case. Coordinates:
\coordinates{(0,0,A), (126,$-$151,B), ($-$193,$-$33,B), (67,184,B), (130,$-$123,A), ($-$171,$-$51,A), (41,174,A), (62,249,A), (184,$-$178,A), ($-$246,$-$70,A), (126,$-$163,A), ($-$204,$-$27,A), (78,190,A), ($-$30,195,B), (183,$-$71,B), ($-$153,$-$123,B), (87,266,B), (185,$-$208,B), ($-$273,$-$57,B), ($-$492,$-$73,A), (183,463,A), (309,$-$390,A), ($-$353,$-$127,A), (67,369,A), (286,$-$242,A).}

The obtained results conclusively confirm the equalities:
\begin{equation}
h(4,0; 3,0)=26, \quad h_{nc}(4,0; 4,0)=26.
\end{equation}

For the instances $h(4,0; 4,3)$, $h(4,0; 4,2)$, and $h_{nc}(4,2; 4,2; 4,2)$, the upper bound of the signotope was not reached, resulting in only interval estimates:
\begin{equation}
20 \le h(4,0; 4,3) \le 21, \quad 23 \le h(4,0; 4,2) \le 25, \quad 21 \le h_{nc}(4,2; 4,2; 4,2) \le 22.
\end{equation}
For instances with a large number of points $N$, where the exact value of $\tilde{h}(\dots)$ and the upper bound remain unknown (lower part of the table), the gap between the estimates may be significantly higher.

In the search for realizations, priority was given to the linear grid of abscissae ($x_i = i$, parameter \texttt{xgrid=1}) to obtain less cumbersome integer solutions. However, in several cases, this strategy proved to be suboptimal. For instance, for the configurations $h(3,2; 3,2; 3,2; 3,2)$ and $h(3,4; 3,4; 3,4; 3,4; 3,4)$, no realization of 17 points was found over several months of computation using the linear grid, whereas the solutions were obtained within a few hours upon switching to grid No. 4.

Similarly, in the analysis of $h(4,0; 4,2)$ and $h(3,3; 3,3; 3,3; 3,3; 3,3)$, two instances of the solver were run in parallel for the first and fourth grids, respectively. In both cases, the computations on grid No. 4 were successful, while the processes on the linear grid were terminated due to timeout. The following section provides a theoretical justification for why the abscissa grid $x_i = i$ demonstrates low efficiency in some scenarios.

Despite the fact that the average complexity of the problem grows exponentially with increasing $N$, the distribution of search time is characterized by high variance. Cases of anomalously fast finding of answers at large $N$ are explained by the algorithm hitting narrow satisfiability zones in the early stages of the search tree traversal, which allows avoiding exhaustive analysis of conflicting configurations.

All found sets of coordinates are available in the project repository\footnote{\url{https://github.com/koshelevv/Erdos-Szekeres/tree/main/colored_points}}.

\subsection{clingo-lpx and Comparative Analysis of Abscissa Grids}

Although the \texttt{clingo-lpx} solver is inferior in performance to Z3, its use is considered appropriate in at least two scenarios:
\begin{itemize}
    \item the application of cardinal constraints in the logical program (for example, specifying a fixed or minimum number of monochromatic empty triangles in the target configuration);
    \item the need to generate the full solution space to collect statistical data (option \texttt{-n0}).
\end{itemize}

For illustration, let us introduce the notations $S_3(N)$, $S_4(N)$, $S_5(N)$ for the sets of signotopes that minimize the number of convex and empty 3-, 4-, and 5-gons on $N$ points, respectively (within the monochromatic model).

Below is the logical program for searching for elements of these sets, as well as their intersections and differences:

\begin{lstlisting}[language=clingo, caption={Program for analyzing extremal properties of signotopes}]
pt(0..n-1).

min_p3(3,1; 4,3; 5,7; 6,13; 7,21; 8,31; 9,43; 10,58; 11,75; 12,94; 13,114).
min_p4(3,0; 4,0; 5,1; 6,3; 7,6; 8,10; 9,15; 10,23; 11,32; 12,42; 13,51).

min_p5(9,0; 10,1; 11,2; 12,3; 13,3; 14,6; 15,9; 16,11).

1{l(A,B,C,R): R=(-1;1)}1 :- pt(A),pt(B),pt(C), A<B,B<C.

:- l(A,B,C,R), l(A,C,D,-R), l(B,C,D,R).
:- l(A,B,C,R), l(A,B,D,-R), l(A,C,D,R).
:- l(A,B,C,R), l(A,B,D,-R), l(B,C,D,R).
:- l(A,B,D,R), l(A,C,D,-R), l(B,C,D,R).
:- l(1,B,C,-1), sb!=off.

i(A,B,C,X) :- l(A,X,B,R), l(A,X,C,-R), B<C.
i(A,B,C,X) :- l(B,X,C,R), l(A,X,C,-R), A<B.

ne(A,B,C) :- i(A,B,C,X).

p3(A,B,C) :- pt(A),pt(B),pt(C), A<B,B<C, not ne(A,B,C).

:- #count{A,B,C: p3(A,B,C)}=M, min_p3(n,M), p3=-1.
:- #count{A,B,C: p3(A,B,C)}>M, min_p3(n,M), p3= 1.

p4(A,B,C,D) :- l(A,B,C,R), l(B,C,D,R), A<B,B<C,C<D, not ne(A,B,C), not ne(A,C,D).
p4(A,B,C,D) :- l(A,B,D,R), l(A,C,D,-R), A<B,B<C,C<D, not ne(A,B,C), not ne(B,C,D).

:- #count{A,B,C,D: p4(A,B,C,D)}=M, min_p4(n,M), p4=-1.
:- #count{A,B,C,D: p4(A,B,C,D)}>M, min_p4(n,M), p4= 1.

p5(A,B,C,D,E) :- l(A,B,C,R), l(B,C,D,R), l(C,D,E,R), A<B,B<C,C<D,D<E, not ne(A,B,C), not ne(A,C,D), not ne(A,D,E).
p5(A,B,C,D,E) :- l(A,B,C,R), l(B,C,E,R), l(A,D,E,-R), A<B,B<C,C<D,D<E, not ne(A,B,C), not ne(A,C,E), not ne(A,D,E).
p5(A,B,C,D,E) :- l(A,B,D,R), l(B,D,E,R), l(A,C,E,-R), A<B,B<C,C<D,D<E, not ne(A,B,D), not ne(A,D,E), not ne(A,C,E).
p5(A,B,C,D,E) :- l(A,C,D,R), l(C,D,E,R), l(A,B,E,-R), A<B,B<C,C<D,D<E, not ne(A,C,D), not ne(A,D,E), not ne(A,B,E).

:- #count{A,B,C,D,E: p5(A,B,C,D,E)}=M, min_p5(n,M), p5=-1.
:- #count{A,B,C,D,E: p5(A,B,C,D,E)}>M, min_p5(n,M), p5= 1.
\end{lstlisting}

One of the central tasks in this field is the determination of the value $X_k(n)$ --- the minimum number of empty $k$-gons among all possible configurations of $n$ points in general position.

A significant milestone in studying this problem was the work of Dehnhardt \cite{Dehnhardt1987}, which first attempted to systematize the search for minimum values for small $n$ and formulated several hypotheses on the relationship between optimal point sets. In particular, the so-called Dehnhardt question was actively discussed in the literature: must a set of points that minimizes the number of convex and empty $k$-gons ($X_k$) also be minimizing for $j$-gons ($X_j$) when $k \neq j$? Dehnhardt assumed the existence of universal extremal configurations; however, his hypotheses regarding the exact values for $n=12$ (specifically, $X_3(12)=95$ and $X_4(12)=44$) were subsequently adjusted.

With the development of computational geometry and the creation of the Order Type Database by Aichholzer and colleagues \cite{Aichholzer2002, Aichholzer2013}, the verification of these assumptions on full samples became possible. It was proved that for $n=12$, the direct connection between minima is broken: there exists a configuration that minimizes the number of pentagons ($X_5(12)=3$), but is not optimal for triangles ($X_3(12)=94$) and quadrilaterals ($X_4(12)=42$).

\begin{theorem}
For $3 \le N \le 11$, the equality $S_3(N) = S_4(N)$ holds. Furthermore, the relationship $S_3(9) = S_4(9) = S_5(9)$ is true.
\end{theorem}

\begin{proof}
To verify this statement, we calculate the differences of the specified pairs of sets using the developed logical program:
\begin{lstlisting}[language=bash]
for n in `seq 3 11`; do
  clingo minimize.lp --configuration=crafty -c p3=1 -c p4=-1 -c n=$n;
  clingo minimize.lp --configuration=crafty -c p3=-1 -c p4=1 -c n=$n;
done

clingo minimize.lp --configuration=crafty -c p4=1 -c p5=-1 -c n=9
clingo minimize.lp --configuration=crafty -c p4=-1 -c p5=1 -c n=9
\end{lstlisting}
For all instances, the verdict \texttt{UNSATISFIABLE} was obtained, indicating the emptiness of the corresponding differences. Thus, any configuration of $3 \le N \le 11$ points that minimizes the number of empty triangles is also minimizing for convex and empty quadrilaterals, and vice versa.
\end{proof}

Verification of this fact for $N > 11$ requires significant computational resources and was not conducted within the scope of this study. Nevertheless, local stochastic search, run for $N = 12$ and $N = 13$ over several days, did not reveal any counterexamples.

For the cases $N = 10$ and $N = 11$, a strict inclusion $S_4(N) \subsetneq S_5(N)$ was established. Using the \texttt{clingo-lpx} extension, we found the coordinates of configurations that achieve the minimum number of convex and empty pentagons ($X_5(10) = 1$ and $X_5(11) = 2$, respectively), while the number of triangles and quadrilaterals is not optimal: for $N = 10$, 59 and 24 were obtained (instead of $X_3 = 58$ and $X_4 = 23$), and for $N = 11$ --- 76 and 33 (instead of $X_3 = 75$ and $X_4 = 32$).

\begin{lstlisting}[language=bash]
clingo minimize.lp --configuration=crafty -c p5=-1 -c p4=1 -c n=10
UNSATISFIABLE
clingo-lpx lpx minimize.lp --configuration=crafty -c p5=1 -c p4=-1 -c n=10 -c xgrid=1
SATISFIABLE
y(0)=2023/3 y(1)=3673/6 y(2)=551 y(3)=459 y(4)=359 y(5)=525/2 y(6)=1577/8 y(7)=90 y(8)=0 y(9)=0

clingo minimize.lp --configuration=crafty -c p5=-1 -c p4=1 -c n=11
UNSATISFIABLE
clingo-lpx lpx minimize.lp --configuration=crafty -c p5=1 -c p4=-1 -c n=11 -c xgrid=1
SATISFIABLE
y(0)=47081/5 y(1)=-169/5 y(2)=361/3 y(3)=2261591/420 y(4)=1 y(5)=-7001731/420
y(6)=-2260751/210 y(7)=45 y(8)=15692/15 y(9)=0 y(10)=0
\end{lstlisting}

To conclude this section, we use the problem of minimizing empty polygons for a comparative analysis of the efficiency of various abscissa grids. Using the \texttt{clingo-lpx} solver, we calculated the number of signotopes from the set $S_3(N) = S_4(N)$ that possess a geometric realization for each grid with the \texttt{xgrid} parameter from 1 to 12 for $4 \le N \le 10$. The results are presented in Table~\ref{tab:grids}.

\begin{lstlisting}[language=bash]
for n in `seq 4 10`; do
  for xgrid in `seq 0 12`; do
    clingo-lpx minimize.lp lpx --configuration=crafty -n0 --quiet=2 --enum-mode=bt -c p4=1 -c n=$n -c xgrid=$xgrid;
  done;
done
\end{lstlisting}

\begin{table}[ht!]
\centering
\begin{tabular}{|l|ccccccc|}
\hline
N & 4 & 5 & 6 & 7 & 8 & 9 & 10 \\
\hline
xgrid=0 & 4 & 22 & 224 & 2604 & 21408 & 31884 & 1937396 \\
xgrid=1 & 4 & 22 & 212 & 2056 & 11876 & 7144 & 165048 \\
xgrid=2 & 4 & 22 & 212 & 2064 & 13128 & 11000 & 335908 \\
xgrid=3 & 4 & 22 & 212 & 2220 & 15304 & 17964 & 598560 \\
xgrid=4 & 4 & 22 & 220 & 2408 & 16640 & 20524 & 688088 \\
xgrid=5 & 4 & 22 & 220 & 2416 & 17076 & 21104 & 718560 \\
xgrid=6 & 4 & 22 & 220 & 2420 & 17160 & 21344 & 725812 \\
xgrid=7 & 4 & 22 & 220 & 2424 & 17164 & 21372 & 727884 \\
xgrid=8 & 4 & 22 & 220 & 2424 & 17172 & 21388 & 729004 \\
xgrid=9 & 4 & 22 & 220 & 2424 & 17172 & 21376 & 729196 \\
xgrid=10 & 4 & 22 & 220 & 2424 & 17172 & 21376 & 729528 \\
xgrid=11 & 4 & 22 & 220 & 2424 & 17172 & 21376 & 729580 \\
xgrid=12 & 4 & 22 & 220 & 2424 & 17172 & 21376 & 729588 \\
\hline
\end{tabular}
\vskip 0.1cm
\footnotesize $xgrid=0$ --- corresponds to the total number of abstract signotopes (without abscissa constraints).
\caption{Number of realizable signotopes from $S_4(N)$ depending on the choice of abscissa grid.}
\label{tab:grids}
\end{table}

According to the data obtained, the uniform abscissa grid ($xgrid=1$) possesses the lowest realizability coverage (realizes the minimum number of signotopes). As the exponential spacing parameter of the grid increases, the number of successfully found realizations monotonically increases, confirming the advantage of non-linear grids in verifying complex configurations.

\section{Convex Hexagons}
\label{sec:results_h}

In this section, the notation $\varhexagon_k$ is used for a convex hexagon containing exactly $k$ points of the set $\mathcal{X}$ in its interior. The study of the properties of such configurations was conducted using Answer Set Programming (ASP) methods.

\begin{proof}[\bf Proof of Theorem \ref{h61}]
To establish the equalities $h(6,2)=17$ and $h(6,1)=18$, we used the ASP model presented below. Technical optimization in this case consisted of checking the number of interior points not in the entire hexagon $ABCDEF$, but only in one of its base triangles (for example, $\triangle ACE$). The running time of the program for parameters $(k=2, n=17)$ and $(k=1, n=18)$ was 70 and 190 minutes, respectively.

\begin{lstlisting}[caption={ASP code for verifying the values of $h(6,k)$}]
pt(1..n).
1{l(A,B,C,R): R=(-1;1)}1 :- pt(A),pt(B),pt(C),A<B,B<C.

% GEOMETRIC CONSTRAINTS
:- l(A,B,C,R), l(A,C,D,-R), l(B,C,D,R).
:- l(A,B,C,R), l(A,B,D,-R), l(A,C,D,R).
:- l(A,B,C,R), l(A,B,D,-R), l(B,C,D,R).
:- l(A,B,D,R), l(A,C,D,-R), l(B,C,D,R).
:- l(1,B,C,-1), sb!=off.

% DEFINITION OF INTERIOR POINTS
i(A,B,C,X) :- l(A,X,B,R), l(A,X,C,-R), B<C.
i(A,B,C,X) :- l(B,X,C,R), l(A,X,C,-R), A<B.

tr(A,B,C) :- pt(A), pt(B), pt(C), A<B,B<C, {i(A,B,C,X)}<=k.

:- l(A,B,C,R), l(B,C,D,R), l(C,D,E,R),  l(D,E,F,R),  A<B,B<C,C<D,D<E,E<F, tr(A,C,E).
:- l(A,B,C,R), l(B,C,D,R), l(C,D,F,R),  l(A,E,F,-R), A<B,B<C,C<D,D<E,E<F, tr(B,D,E).
:- l(A,B,C,R), l(B,C,E,R), l(C,E,F,R),  l(A,D,F,-R), A<B,B<C,C<D,D<E,E<F, tr(B,D,E).
:- l(A,B,D,R), l(B,D,E,R), l(D,E,F,R),  l(A,C,F,-R), A<B,B<C,C<D,D<E,E<F, tr(B,C,E).
:- l(A,C,D,R), l(C,D,E,R), l(D,E,F,R),  l(A,B,F,-R), A<B,B<C,C<D,D<E,E<F, tr(B,C,E).
:- l(A,B,C,R), l(B,C,F,R), l(A,D,E,-R), l(D,E,F,-R), A<B,B<C,C<D,D<E,E<F, tr(A,C,E).
:- l(A,B,D,R), l(B,D,F,R), l(A,C,E,-R), l(C,E,F,-R), A<B,B<C,C<D,D<E,E<F, tr(A,D,E).
:- l(A,C,D,R), l(C,D,F,R), l(A,B,E,-R), l(B,E,F,-R), A<B,B<C,C<D,D<E,E<F, tr(A,D,E).

#show l/4.
\end{lstlisting}

\vskip+1cm

\begin{lstlisting}[language=bash, caption={Verification logs in clingo}]
clingo ES_hexagons.lp --configuration=frumpy --sat-p=3 -c k=2 -c n=16
SATISFIABLE
clingo ES_hexagons.lp --configuration=frumpy --sat-p=3 -c k=2 -c n=17
UNSATISFIABLE
clingo ES_hexagons.lp --configuration=frumpy --sat-p=3 -c k=1 -c n=17
SATISFIABLE
clingo ES_hexagons.lp --configuration=frumpy --sat-p=3 -c k=1 -c n=18
UNSATISFIABLE
\end{lstlisting}

Using the linear subreduction method (with $x_i = i$), we successfully constructed an example of 17 points containing neither $\varhexagon_0$ nor $\varhexagon_1$ in just a few minutes. A similar example was provided in \cite{h61_comp}, but its manual search at that time took several weeks. The coordinates of the found configuration:
\begin{center}
{\small (0,$-$114449), (1,$-$193125), (2,$-$98112), (3,$-$90290), (4,$-$102071), (5,$-$496), (6,$-$769), (7,115376), (8,96152), (9,$-$8702), (10,662056), (11,347088), (12,32056), (13,0), (14,0), (15,8206), (16,192)}
\end{center}
\end{proof}

Note that from the proven equality $h(6,2)=17$, the values $h(6,k)=17$ follow immediately for all $k>2$, since the condition of having exactly $k$ points inside a hexagon for $k \le 2$ is a more rigid constraint for a set of 17 points.

Since $h(6,2)=17$, any set of 17 points in the plane must contain at least one hexagon $\varhexagon_k$ for $k \in \{0, 1, 2\}$. We studied the question of the existence of configurations of 17 points containing a \textbf{unique} convex hexagon of a specific type:

\begin{itemize}
\item For the case $\varhexagon_0$, the answer is positive; an example is easily constructed based on the classical configuration for $g(6)=17$: {\small (0,0), (9,1), (20,2), (30,2), (41,1), (50,0), (60,48), (65,49), (70,48), (80,52), (85,51), (90,52), (0,99), (9,98), (20,97), (30,97), (41,98)}
\item For the case $\varhexagon_1$, the answer is also positive. The configuration was found by us using the proposed method: {\small (0,$-$6091), (1,$-$1), (2,0), (3,$-$4504), (4,315), (5,109787), (6,$-$1771), (7,73098), (8,$-$2), (9,48726), (10,44), (11,$-$276), (12,22), (13,$-$13), (14,1881), (15,2339), (16,$-$18)}.
\item For the case $\varhexagon_2$, the answer is negative, which strictly follows from the structural results provided below.
\end{itemize}

Below are the coordinates of configurations containing \textbf{only} the specified types of hexagons:

{\bf $\{\varhexagon_1, \varhexagon_2\}$ --- 18 points:}
{\small (0,1905), (1,1419), (2,937), (3,$-$13196), (4,$-$12255), (5,$-$1422), (6,$-$1383), (7,$-$1341), (8,$-$38634), (9,0), (10,$-$6802), (11,$-$3412), (12,$-$78970), (13,$-$4077), (14,$-$99211), (15,$-$1268), (16,$-$971), (17,0)}.

{\bf $\{\varhexagon_1, \varhexagon_2, \varhexagon_3\}$ --- 19 points:}
{\small (172,$-$83), (170,112), (82,$-$123), (36,$-$42), ($-$191,29), (250,$-$294), (135,$-$69), (249,160), (15,$-$58), (89,$-$115), (102,9), (209,$-$208), ($-$260,32), (135,20), (239,153), (114,117), (296,204), ($-$110,32), (35,$-$41)}.

{\bf $\{\varhexagon_1, \varhexagon_2, \varhexagon_3, \varhexagon_4\}$ --- 20 points:}
{\small (140,33), ($-$194,95), ($-$29,$-$30), (209,$-$218), (106,115), (31,28), (79,166), (204,228), (140,56), (201,191), ($-$250,100), (165,12), (229,266), ($-$167,88), (124,69), (232,$-$273), (154,$-$139), (94,137), (194,188), (186,186)}.

\subsection{The Fourth Erdős--Szekeres-type Problem}

\textbf{The Fourth Erdős--Szekeres-type Problem.} \textit{For integers $n \ge 3$ and $q \ge 2$, find the smallest positive integer $h_{mod}(n,q)$ such that any set of points ${\cal X}$ in the plane in general position with cardinality $|{\cal X}| \ge h_{mod}(n,q)$ contains the vertices of a convex $n$-gon $C$ for which $|(C \setminus \partial C) \cap {\cal X}| \equiv 0 \pmod q$.}

The Bialostocki--Dierker--Voxman conjecture \cite{BDV} states that the value $h_{mod}(n, q)$ exists for all $n \ge 3, q \ge 2$. The authors proved it for the case $n \ge q+2$, establishing an upper bound via Ramsey numbers:
\begin{equation}
h_{mod}(n, q) \le g(R_3(n', n', \dots, n')),
\label{BDV_est}
\end{equation}
where $n' \ge n$ and $n' \equiv 2 \pmod q$.

The bound obtained by Caro \cite{Caro} for points with weights from an Abelian group, applied to this problem, also exhibits a tower of exponents character due to its dependence on the Ramsey numbers $R_2$:
\begin{equation}
h_{mod}(n, q) \le g\left(\left(R_2(3q-3, \dots, 3q-3) + 1\right) \left(\left\lfloor \frac{n}{q} \right\rfloor + 1\right) q\right).
\label{Caro_est}
\end{equation}

Further research aimed to refine these estimates and relax the condition $n \ge q+2$. Károlyi, Pach, and Tóth \cite{KPT} showed the existence of $h_{mod}(n, q)$ for $n \ge 5q/6 + O(1)$, although this result did not improve the exponential nature of the bounds.

In 2011, one of the authors \cite{mod} improved the technique of Bialostocki, Dierker, and Voxman:

\begin{theorem} \label{mod_tech}
If $n \ge q+2$, then for even and odd $q$, respectively:
$$h_{mod}(n, q) \le R_3(n, n, \dots, n), \quad h_{mod}(n, q) \le R_3(g(n), n, \dots, n).$$
\end{theorem}

The main result of \cite{mod} was the complete elimination of the dependence on Ramsey numbers under a slightly stronger constraint on $n$:

\begin{theorem} \label{mod_best}
If $n \ge 2q-1$, then $h_{mod}(n, q) \le g(q(n-4)+4)$.
\end{theorem}

Since $g(q(n-4)+4) \le 2^{qn+O(1)}$, this theorem is significantly more efficient than all previous results, as it removes the multiple exponents from the final expression.

\subsubsection{New Results}

To prove stronger statements and optimize logical programs (by transitioning from counting points in hexagons to analyzing their constituent quadrilaterals), we formulate two auxiliary problems:

\begin{enumerate}
\item $h_{ex}(n, Q)$ --- the minimum number of points guaranteeing the existence of a convex $n$-gon whose number of interior points belongs to the set $Q \subset \mathbb{N}$ ($0 \in Q$).
\item $h_{sub}(6, q)$ --- the minimum number of points guaranteeing the existence of a convex hexagon in which at least one of the three quadrilaterals, obtained by removing a pair of opposite vertices, contains exactly $0$ or $q$ interior points.
\end{enumerate}

For signotopes, the corresponding functions are denoted as $\tilde{h}_{mod}(n,q), \tilde{h}_{ex}(n,Q), \tilde{h}_{sub}(6,q)$. The following chains of inequalities are obvious:
\begin{gather*}
h_{mod}(6,q) \le h_{ex}(6,\{0,q\}) \le h_{sub}(6,q) \le h(6)=30; \\
\tilde{h}_{mod}(6,q) \le \tilde{h}_{ex}(6,\{0,q\}) \le \tilde{h}_{sub}(6,q) \le \tilde{h}(6)=30.
\end{gather*}

\begin{theorem}
$h_{sub}(6,2)=\tilde{h}_{sub}(6,2)=18, \quad h_{sub}(6,3)=\tilde{h}_{sub}(6,3)=20, \quad h_{sub}(6,4)=\tilde{h}_{sub}(6,4)=21$.
\end{theorem}

\begin{proof}
The logical program for verifying the values of $\tilde{h}_{sub}(6,q)$ and the logs of the runs with the status \texttt{UNSATISFIABLE} are available in the project repository\footnote{\url{https://github.com/koshelevv/Erdos-Szekeres/tree/main/hexagons/original_sub4}}. The total computation time (in CPU Hours) was:
\begin{lstlisting}[language=bash, basicstyle=\small\ttfamily]
hexagons.sub4.q2.18.log: CPU Time : 16.70 h
hexagons.sub4.q3.20.log: CPU Time : 313.95 h
hexagons.sub4.q4.21.log: CPU Time : 5135.51 h
\end{lstlisting}

Based on the analysis of the previously constructed example of 17 points with a unique hexagon of type $\varhexagon_1$, the equality $h_{mod}(6,2) = h_{ex}(6,\{0,2\}) = h_{sub}(6,2)=18$ is established.

To confirm that $h_{sub}(6,3)=20$ and $h_{sub}(6,4)=21$, the \textbf{linear subreduction method} was applied on an exponential abscissa grid. The coordinates of geometric realizations for the corresponding extremal configurations were found:

\textbf{19 points ($q=3$):}
{\small ($-$65536, $-$6779576), ($-$16384, $-$1705174), ($-$4096, $-$426607), ($-$1024, $-$306602), ($-$256, $-$29548), ($-$64, 18492168), ($-$16, 564027), ($-$4, 565381), ($-$1, 566107), (0, $-$301884), (1, $-$2975), (4, $-$1), (16, $-$1421), (64, 3569), (256, $-$256602), (1024, $-$120759), (4096, 422606), (16384, 0), (65536, 7);}

\textbf{20 points ($q=4$):}
{\small ($-$262144, $-$110327839), ($-$65536, 131072), ($-$16384, 47544316), ($-$4096, 39608554), ($-$1024, 36175252), ($-$256, 63101785), ($-$64, 36914751), ($-$16, 36929496), ($-$4, 35620882), ($-$1, 36952837), (1, 72112362), (4, 94007251), (16, 21270346), (64, 36989358), (256, 37133021), (1024, 36301986), (4096, 34612949), (16384, 26548716), (65536, 0), (262144, $-$110761442).}
\end{proof}

Using the ASP solver, abstract signotopes on 19 elements were found containing the subset of types $\{\varhexagon_1, \varhexagon_2, \varhexagon_4\}$, and on 20 elements, containing $\{\varhexagon_1, \varhexagon_2, \varhexagon_3, \varhexagon_5\}$. These examples allow establishing the following exact values of the combinatorial functions:
\begin{gather*}
\tilde{h}_{mod}(6,3) = \tilde{h}_{ex}(6,\{0,3\}) = \tilde{h}_{sub}(6,3) = 20, \\
\tilde{h}_{mod}(6,4) = \tilde{h}_{ex}(6,\{0,4\}) = \tilde{h}_{sub}(6,4) = 21.
\end{gather*}

Since during the computational experiment (totaling more than 12 months of CPU time), the linear subreduction method failed to construct geometric realizations for signotopes excluding both $\{\varhexagon_0, \varhexagon_3\}$ at $N=19$ and $\{\varhexagon_0, \varhexagon_4\}$ at $N=20$, only interval estimates are currently valid for the plane:
\[ 19 \le h_{mod}(6,3) \le 20, \quad 20 \le h_{mod}(6,4) \le 21. \]

\begin{theorem}
The following values hold for the existence functions of convex hexagons with a given set of interior points:
\begin{gather*}
\tilde{h}_{ex}(6,\{0,1,2\}) = \tilde{h}_{ex}(6,\{0,1,3\}) = \tilde{h}_{ex}(6,\{0,1,4\}) = g(6) = 17; \\
h_{ex}(6,\{0,3,4\}) = \tilde{h}_{ex}(6,\{0,3,4\}) = 19, \quad h_{ex}(6,\{0,4,5\}) = \tilde{h}_{ex}(6,\{0,4,5\}) = 20.
\end{gather*}
\end{theorem}

\begin{proof}
The logical program for verifying $\tilde{h}_{ex}(6,Q)$ and the corresponding logs (\texttt{UNSATISFIABLE}) are presented in the project repository\footnote{\url{https://github.com/koshelevv/Erdos-Szekeres/tree/main/hexagons/original_ex}}:
\begin{lstlisting}[language=bash, basicstyle=\small\ttfamily]
hexagons.Q_012.17pt.log: CPU Time : 4.13 h
hexagons.Q_013.17pt.log: CPU Time : 6.07 h
hexagons.Q_014.17pt.log: CPU Time : 10.86 h
hexagons.Q_034.19pt.log: CPU Time : 188.42 h
hexagons.Q_045.20pt.log: CPU Time : 2227.30 h
\end{lstlisting}
\end{proof}

An analysis of the signotope space and the found realizations allows us to formulate a structural theorem describing the mandatory presence of certain types of hexagons depending on the cardinality of the set $\mathcal{X}$.

\begin{theorem}[on the structure of small configurations]
Let $\mathcal{X}$ be a set of points in the plane in general position. Then the following statements hold:
\begin{enumerate}
\item If $|\mathcal{X}| = 17$, then $\mathcal{X}$ contains either $\varhexagon_0$, or $\varhexagon_1$, or simultaneously $\{\varhexagon_2, \varhexagon_3, \varhexagon_4\}$.
\item If $|\mathcal{X}| = 18$, then $\mathcal{X}$ contains either $\varhexagon_0$, or simultaneously $\{\varhexagon_1, \varhexagon_2\}$.
\item If $|\mathcal{X}| = 19$, then $\mathcal{X}$ contains either $\varhexagon_0$, or simultaneously $\{\varhexagon_1, \varhexagon_2\}$ and at least one hexagon from $\{\varhexagon_3, \varhexagon_4\}$.
\item If $|\mathcal{X}| = 20$, then $\mathcal{X}$ contains either $\varhexagon_0$, or simultaneously $\{\varhexagon_1, \varhexagon_2, \varhexagon_3\}$ and at least one hexagon from $\{\varhexagon_4, \varhexagon_5\}$.
\item If $|\mathcal{X}| = 21$, then $\mathcal{X}$ contains either $\varhexagon_0$, or simultaneously $\{\varhexagon_1, \varhexagon_2, \varhexagon_3, \varhexagon_4\}$.
\end{enumerate}
\end{theorem}

\textbf{Corollary.} {\it There is no set of 17 points in general position containing $\varhexagon_2$ as the unique convex hexagon.}

\vskip 0.2cm
\noindent \textbf{Note.} The formulated theorems are strictly proven for abstract signotopes. The fact that, during an extensive computational experiment, we failed to find a geometric realization of {\it anomalous} signotopes (for example, those containing $\varhexagon_4$ in the absence of $\varhexagon_0$ and $\varhexagon_3$ for $N=19$) indicates that stronger statements may hold for point sets in the plane. In particular, we conjecture the mandatory presence of either $\varhexagon_0$ or $\varhexagon_3$ as early as $N=19$. Investigating this gap between the combinatorial structure of abstract and realizable configurations remains an open problem.

\subsubsection{Symmetric Configurations and Lower Bounds for $h_{mod}(6, q)$}

In concluding this section, let us consider symmetric point sets. We have established that any set of $N \ge 18$ points possessing axial symmetry is guaranteed to contain a convex empty hexagon $\varhexagon_0$. The low value of this upper bound makes axially symmetric configurations ineffective when searching for extremal examples for the function $h_{mod}(6, q)$.

To verify this fact, we use our logical program for $h(6,k)$, adding a single line that restricts the search space to only axially symmetric signotopes:

\begin{lstlisting}[language=bash]
for n in 17 18 19; do
    (cat ES_hexagons.lp; echo ':- l(A,B,C,R), l(n+1-C,n+1-B,n+1-A,-R).') | clingo -c k=0 -c n=$n -c sb=off;
done
SATISFIABLE
UNSATISFIABLE
UNSATISFIABLE
\end{lstlisting}

To construct examples with a larger number of points, more flexible structures were used, particularly configurations with $k$-fold symmetry. Using the local stochastic search method, symmetric sets for large $N$ were found containing exclusively hexagons of types $\{\varhexagon_1, \dots, \varhexagon_{q-1}\}$. The corresponding coordinates and visualizations are available in the project repository\footnote{\url{https://github.com/koshelevv/Erdos-Szekeres/tree/main/hexagons}}. These examples allow us to establish lower bounds for the function $h_{mod}(6, q)$, presented in the table below. For these examples, we do not guarantee that $N$ is the maximum possible for a fixed $q$.

{\bf $\{\varhexagon_1, \dots, \varhexagon_6\}$ --- 21 points:}
{\small ($-$76,$-$26) (15,78) (60,$-$52) ($-$22,229) (209,$-$95) ($-$187,$-$133) ($-$26,218) (201,$-$86) ($-$175,$-$131) ($-$35,182) (175,$-$60) ($-$140,$-$121) ($-$209,$-$217) ($-$83,289) (292,$-$72) ($-$300,$-$249) ($-$65,384) (365,$-$135) ($-$174,$-$127) ($-$22,214) (196,$-$87).}

{\bf $\{\varhexagon_1, \dots, \varhexagon_7\}$ --- 22 points:}
{\small (0,0) ($-$2,$-$117) ($-$100,60) (102,56) (4,$-$169) ($-$148,81) (144,87) (17,$-$78) ($-$76,24) (59,53) ($-$107,40) (88,72) (18,$-$112) ($-$252,$-$223) ($-$67,329) (319,$-$106) ($-$273,$-$241) ($-$72,356) (345,$-$115) (1,$-$83) ($-$72,40) (71,42).}

\begin{table}[h]
\centering
\small
\begin{tabular}{|c|ccccccccccccccccc|}
\hline
$q$ & 2 & 3 & 4 & 5 & 6 & 7 & 8 & 9 & 10 & 11 & 12 & 13 & 14 & 15 & 16 & 17 & 18 \\ \hline
$h_{mod}(6, q) \ge$ & 18 & 19 & 20 & 21 & 21 & 22 & 23 & 23 & 23 & 23 & 23 & 23 & 26 & 26 & 26 & 26 & 30 \\ \hline
\end{tabular}
\caption{Lower bounds for the fourth Erdős--Szekeres type problem ($n=6$).}
\label{tab:mod_bounds}
\end{table}

It should be noted that the value of 26 for $14 \le q \le 17$ is due to the example we found of 25 points with 3-fold symmetry for the problem $h_{nc}(4,0; 3,0)$, containing only types $\{\varhexagon_1, \dots, \varhexagon_{13}\}$. For $q \ge 18$, the estimates coincide with the classical value $h(6) = 30$, as the well-known Overmars configuration of 29 points \cite{Over} contains only hexagons of types $\{\varhexagon_1, \dots, \varhexagon_{17}\}$.

\section{Geometric Ramsey Numbers}

Research in this area was initiated by Bárány and Károlyi \cite{BK_triangles}, who formulated the following fundamental problem:

\begin{quote}
{\it Is it true that for any integer $c$, there exists a minimum positive integer $N$ such that for any set of $N$ points in the plane in general position and an arbitrary $c$-coloring of the edges of the complete graph with vertices at these points, there necessarily exists a monochromatic empty triangle?}
\end{quote}

Batista-Santiago et al. in \cite{BS_general} gave a comprehensive negative answer for the case $c \ge 3$, proving that the sought number $R_{EC}(3, 3, 3)$ (and above) does not exist. For two colors ($c=2$), they established the interval $17 \le R_{EC}(3, 3) \le h(6)=30$. They also proposed a generalization of the problem for polygons of arbitrary size: finding the minimum number $R_{EC}(s, t)$ guaranteeing, in any two-color edge coloring, the existence of an empty convex $s$-gon of the first color or an empty convex $t$-gon of the second color. The authors obtained the lower bound $57 \le R_{EC}(4, 4)$.

In the present study, we introduce additional values for further classification of conditions: $R_{C}(s, t)$ (where the emptiness condition is not mandatory), as well as $R_{ENC}(s, t)$ and $R_{NC}(s, t)$ (where the convexity condition is not mandatory). From the definitions, the following relations follow directly:
\begin{equation}
R_{EC}(2, t) = h(t), \quad R_{C}(2, t) = g(t), \quad R_{ENC}(2, t) = R_{NC}(2, t) = t, \quad R_{C}(3, 3) = R_{NC}(3, 3) = 6.
\end{equation}

Calculating the value of $R_{NC}(s, t)$ is equivalent to the problem in ordered graph theory of searching for monochromatic cycles without self-intersections. The exact value $R_{NC}(s, t) = 2st - 3s - 3t + 6$ was established in \cite{K_arcs} (upper bound) and \cite{B_ordered} (lower bound). We have noticed that the argumentation from \cite{K_arcs} allows extending this result to broader structures — kipas graphs — while maintaining the original formula.

Using the linear subreduction method, we obtained new lower bounds for the case of monochromatic convex polygons without the additional emptiness condition: $R_C(3, 4) \ge 11$, $R_C(4, 4) \ge 23$, and $R_C(3, 5) \ge 25$. The corresponding point configurations are presented below:

\noindent\textit{Example for $R_C(3, 4)$ (10 points):}\\ {\small
(0,0) (1,0) (2,$-211$) (3,$-421$) (4,$-671$) (5,$-839$) (6,$-1044$) (7,$-1248$) (8,$-1458$) (9,$-1655$)}

\noindent\textit{Example for $R_C(4, 4)$ (22 points):} 
\coordinates{(0,605219576) (1,547491917) (2,$-38432236$) (3,$-37864720$) (4,$-8094$) (5,327037131) (6,$-22919212$) (7,$-19069521$) (8,$-35026470$) (9,$-43780784$) (10,$-7497592$) (11,$-3645043$) (12,218931) (13,0) (14,$-3332$) (15,13464851) (16,17507537) (17,$-4512052$) (18,23345651) (19,$-15231$) (20,$-1528709$) (21,0)}

\noindent\textit{Example for $R_C(3, 5)$ (24 points):}
\coordinates{(0,$-184$) (1,$-5$) (2,71) (3,19085) (4,45983) (5,588939) (6,0) (7,$-176$) (8,153581) (9,8474) (10,7536) (11,$-281760$) (12,5655) (13,4707) (14,$-717104$) (15,3346) (16,2669) (17,1995) (18,1325) (19,$-774$) (20,$-12$) (21,$-2886$) (22,12) (23,22)}

Calculating signotopic upper bounds for these problems is an extremely labor-intensive process. During the computational experiment using the {\tt clingo} system, we were only able to verify the value $R_C(3, 4) = 11$. Regarding the other two values, no solution was obtained even after 12 months of continuous solver operation, which leads us to put forward the following hypothesis: $R_C(4, 4) = 24$, $23 \le R_C(3, 5) \le 25$.

The central result of this section is the following theorem.
\begin{theorem}
$R_{EC}(3, 3) = 21$.
\end{theorem}
\begin{proof}
To prove the upper bound, the following logical program was used:
\begin{lstlisting}[language=clingo]
pt(1..n).

ins(tr,t1,tr1;tr,t2,tr2;tr,t3,tr3).

1{c(A,B,Z): ins(_,Z,I),I=0..99}1 :- pt(A),pt(B), A<B.
1{l(A,B,C,R): R=(-1;1)}1 :- pt(A),pt(B),pt(C), A<B,B<C.

% GEOMETRIC CONSTRAINTS
:- l(A,B,C,R), l(A,C,D,-R), l(B,C,D,R).
:- l(A,B,C,R), l(A,B,D,-R), l(A,C,D,R).
:- l(A,B,C,R), l(A,B,D,-R), l(B,C,D,R).
:- l(A,B,D,R), l(A,C,D,-R), l(B,C,D,R).
:- l(1,B,C,-1), sb!=off.

% DEFINITION OF INTERIOR POINTS
i(A,B,C,X) :- l(A,X,B,R), l(A,X,C,-R), B<C.
i(A,B,C,X) :- l(B,X,C,R), l(A,X,C,-R), A<B.

tr(A,B,C,I) :- pt(A),pt(B),pt(C), A<B,B<C, {i(A,B,C,X)}<=I, ins(_,_,I), I=0..99.

:- ins(tr,Z,I), A<B,B<C, c(A,B,Z), c(A,C,Z), c(B,C,Z), tr(A,B,C,I).

#show l/4.
#show c/3.
\end{lstlisting}

\begin{lstlisting}[language=bash]
clingo ES_Ramsey.lp --configuration=frumpy --sat-p=3 -c tr1=0 -c tr2=0 -c n=21
UNSATISFIABLE
CPU Time : 1356782.996s
\end{lstlisting}

We found a configuration of 20 points exhibiting 5-fold symmetry. For this set, there exists a symmetric edge coloring that does not contain monochromatic empty triangles. Below are the obtained integer coordinates\footnote{the edge coloring is available on \url{https://github.com/koshelevv/Erdos-Szekeres/tree/main/geometric_Ramsey_numbers}}:
\begin{center}
($-$1583,$-$2563) (1948,$-$2298) (2787,1143) ($-$226,3004) ($-$2927,714) \\
($-$787,$-$966) (676,$-$1047) (1204,319) (69,1244) ($-$1162,450) \\
($-$661,731) ($-$899,$-$403) (105,$-$980) (964,$-$203) (491,855) \\
($-$381,110) ($-$222,$-$328) (244,$-$313) (373,135) ($-$13,396).
\end{center}
\end{proof}

Further development of the problem involves relaxing the emptiness condition for monochromatic triangles. Let $R_{EC}(3, k_1; 3, k_2)$ be the minimum number of points guaranteeing the presence of either a first-color triangle with at most $k_1$ interior points or a second-color triangle with at most $k_2$ interior points. The case $k=0$ corresponds to the classical emptiness condition, while $k = \infty$ completely removes the restriction on the number of interior points.

We have established the following exact values for these quantities:
\begin{itemize}
\item $R_{EC}(3, 0; 3, 1) = R_{EC}(3, 0; 3, \infty) = 17$;
\item $R_{EC}(3, 1; 3, 1) = 8$;
\item $R_{EC}(3, 1; 3, 2) = R_{EC}(3, 1; 3, \infty) = 7$;
\item $R_{EC}(3, 2; 3, 2) = R(3, 3) = 6$.
\end{itemize}

It should be noted that the value $R_{EC}(3, 2; 3, 2) = 6$ expectedly coincides with the classical Ramsey number $R(3,3)$. It is known that in any two-color edge coloring of $K_6$, there are at least two monochromatic triangles. In this case, only one of them can contain the maximum possible 3 interior points, while the second is guaranteed to contain no more than two.

\begin{lstlisting}[float, floatplacement=H, language=bash]
clingo-lpx ES_Ramsey.lp lpx --configuration=frumpy --sat-p=3 -c tr1=0 -c tr2=99 -c n=16
SATISFIABLE
y(1)=-2833/2 y(2)=-10409/8 y(3)=-4739/4 y(4)=1 y(5)=-954 y(6)=-43327/56 y(7)=-4818/7 y(8)=-4815/8
y(9)=-607 y(10)=8509/3 y(11)=59563/18 y(12)=2476273/648 y(13)=-801/4 y(14)=-1581/16 y(15)=0 y(16)=0

clingo-lpx ES_Ramsey.lp lpx --configuration=frumpy --sat-p=3 -c tr1=1 -c tr2=1 -c n=7
SATISFIABLE
y(1)=3/2 y(2)=1 y(3)=-1 y(4)=-13 y(5)=-6 y(6)=0 y(7)=0

clingo-lpx ES_Ramsey.lp lpx --configuration=frumpy --sat-p=3 -c tr1=1 -c tr2=99 -c n=6
SATISFIABLE
y(1)=5 y(2)=7/2 y(3)=-1 y(4)=3/2 y(5)=0 y(6)=0

clingo ES_Ramsey.lp --configuration=frumpy --sat-p=3 -c tr1=0 -c tr2=1 -c n=17
UNSATISFIABLE
clingo ES_Ramsey.lp --configuration=frumpy --sat-p=3 -c tr1=1 -c tr2=1 -c n=8
UNSATISFIABLE
clingo ES_Ramsey.lp --configuration=frumpy --sat-p=3 -c tr1=1 -c tr2=2 -c n=7
UNSATISFIABLE
\end{lstlisting}

\section{Conclusion}
\label{sec:concl}

In this paper, we have presented a comprehensive study of Erdős--Szekeres type problems using modern computational combinatorial analysis methods. The main result of the study is the establishment of several new exact values for the functions $h, h_{nc}$, and $h_{isl}$ for bicolored and multicolored point sets. In particular, the equality $h_{nc}(4,0; 4,0) = 26$ was proven for the first time.

Special attention was given to the convex hexagon problem. New exact values and coordinates were found for configurations with specific constraints on the number of interior points (functions $h_{mod}$, $h_{ex}$, and $h_{sub}$).

An important achievement in the field of geometric Ramsey numbers was the establishment of the exact value $R_{EC}(3,3) = 21$. The discovered configuration of 20 points with 5-fold symmetry and the verification of the upper bound using a SAT solver eliminate the previously existing gap within this range. Furthermore, we have obtained new lower bounds for Ramsey numbers without the emptiness condition: $R_C(4,4) \ge 23$ and $R_C(3,5) \ge 25$, expanding the understanding of the structure of extremal planar sets.

The theoretical analysis of the sets $S_k(N)$ allowed for the refinement of Dehnhardt's hypotheses regarding universal extremal configurations. We established the coincidence of configuration sets minimizing the number of convex and empty 3- and 4-gons for $N \le 11$ and their divergence from analogous sets for 5-gons for $N \ge 10$.

The methodological value of the work lies in the testing and comparison of different approaches to the logical encoding of geometric constraints. We have shown that the use of the ASP system {\tt clingo} in combination with specialized heuristics (presets {\it frumpy} and {\it crafty}) allows for the effective verification of complex configurations in their entirety, minimizing the need for manual decomposition. The developed linear subreduction method using exponential abscissa grids proved its efficiency, enabling the discovery of geometric realizations for the vast majority of found maximal signotopes.

Despite the progress achieved, the question of the existence and exact values of $h(4,0; 4,0)$ and $h_{isl}(4,0; 4,0)$ remains open. We see the further development of the proposed approach in the automation of decomposition processes and the integration of more powerful conflict-learning methods into SMT models, which could potentially overcome the current computational barriers for problems with $N > 30$ points.


\newpage

\begin{thebibliography}{99}

\bibitem{ES} P. Erdős, G. Szekeres, {\it A combinatorial problem in geometry}, Compositio Math., 2 (1935), 463--470.

\bibitem{Low} P. Erdős, G. Szekeres, {\it On some extremum problems in elementary geometry}, Ann. Univ. Sci. Budapest Eötvös Sect. Math., 3--4 (1961), 53--62.

\bibitem{E} P. Erdős, {\it Some more problems in elementary geometry}, Austral. Math. Soc. Gaz., 5 (1978), 52--54.

\bibitem{Sol} W. Morris, V. Soltan, {\it The Erdős--Szekeres problem on points in convex position}, Bulletin of the Amer. Math. Soc., 37     (2000), N4, 437--458.

\bibitem{Ram} F. P. Ramsey, {\it On a problem of formal logic}, Proc. London Math. Soc. Ser. 2, 30 (1930), 264--286.

\bibitem{TRam} R. L. Graham, B. L. Rothschild, J. H. Spencer, {\it Ramsey Theory}, 2nd ed., John Wiley \& Sons, NY, 1990.

\bibitem{Hall} M. Hall, Jr., {\it Combinatorial Theory}, Blaisdell, Waltham, Mass. 1967; Mir, Moscow, 1970.

\bibitem{chrom} O. Devillers, F. Hurtado, G. Károlyi, C. Seara, {\it Chromatic variants of the Erdős--Szekeres theorem}, Comput. Geom., 26 (2003), 193--208.

\bibitem{SL} G. Szekeres, L. Peters, {\it Computer solution to the 17-point Erdős--Szekeres problem}, ANZIAM J., 48 (2006), 151--164.

\bibitem{maric} F. Marić, {\it Fast formal proof of the Erdős--Szekeres conjecture for convex polygons with at most 6 points}, J. Autom. Reason., 62 (2019), 301--329.

\bibitem{scheucher2020} M. Scheucher, {\it Points, Lines, and Circles. Some Contributions to Combinatorial Geometry}, Ph.D. thesis, TU Berlin, 2020.

\bibitem{CG98} F. R. K. Chung, R. L. Graham, {\it Forced convex $n$-gons in the plane}, Discrete Comput. Geom., 19(3) (1998), 367--371.

\bibitem{KP98} D. Kleitman, L. Pachter, {\it Finding convex sets among points in the plane}, Discrete Comput. Geom., 19(3) (1998), 405--410.

\bibitem{TV98} G. Tóth, P. Valtr, {\it Note on the Erdős--Szekeres theorem}, Discrete Comput. Geom., 19(3) (1998), 457--459.

\bibitem{TV05} G. Tóth, P. Valtr, {\it The Erdős--Szekeres theorem: upper bounds and related results}, Combinatorial and Computational Geometry, MSRI Publications, 52 (2005), 557--568.

\bibitem{Vl15} G. Vlachos, {\it On a conjecture of Erdős and Szekeres}, arXiv:1505.07549 [math.CO], 2015.

\bibitem{MV15} H. N. Mojarrad, G. Vlachos, {\it An improved upper bound for the Erdős--Szekeres theorem}, arXiv:1510.06255 [math.CO], 2015.

\bibitem{NY15} S. Norin, Y. Yuditsky, {\it Erdős--Szekeres without induction}, Discrete Comput. Geom., 55 (2016), 963--971.

\bibitem{Suk17} A. Suk, {\it On the Erdős--Szekeres convex polygon problem}, J. Amer. Math. Soc., 30(4) (2017), 1047--1053.

\bibitem{HMPT20} A. F. Holmsen, H. Mojarrad, J. Pach, G. Tardos, {\it Two extensions of the Erdős--Szekeres theorem}, J. Combin. Theory Ser. A, 170 (2020), 105132.

\bibitem{Harb} H. Harborth, {\it Konvexe Fünfecke in ebenen Punktmengen}, Elem. Math., 33 (1978), 116--118.

\bibitem{Over0} M. Overmars, B. Scholten, I. Vincent, {\it Sets without empty convex 6-gons}, Bull. EATCS, 7 (1989), 160--168.

\bibitem{Over} M. Overmars, {\it Finding sets of points without empty convex 6-gons}, Discrete Comput. Geom., 29 (2003), 153--158.

\bibitem{Nicolas} C. Nicolas, {\it The empty hexagon theorem}, Discrete Comput. Geom., 38(2) (2007), 389--397.

\bibitem{Gerken} T. Gerken, {\it On empty convex hexagons in planar point sets}, Discrete Comput. Geom., 39 (2008), 239--272.

\bibitem{Valtr} P. Valtr, {\it On the empty hexagons}, Manuscript, 2006. URL: http://cuni.cz

\bibitem{MSb} V. A. Koshelev, {\it The Erdős--Szekeres problem on empty hexagons in the plane}, Modeling and Analysis of Information Systems, 16:2 (2009), 22--74 (in Russian).

\bibitem{h6} M. J. H. Heule, M. Scheucher, {\it Happy ending: An empty hexagon in every set of 30 points}, arXiv:2403.00737 [math.CO], 2024.

\bibitem{cadical} A. Biere, K. Fazekas, M. Fleury, N. Heisinger, {\it CaDiCaL, Kissat, Paracooba at the SAT Competition 2020}, SAT Competition (2020), 51--53.

\bibitem{h6_formal} B. Subercaseaux, W. Nawrocki, J. Gallicchio, C. Codel, M. Carneiro, M. J. H. Heule, {\it Formal Verification of the Empty Hexagon Number}, arXiv:2403.17370 [cs.LO], 2024.

\bibitem{Hort} J. D. Horton, {\it Sets with no empty 7-gons}, Canad. Math. Bull., 26 (1983), 482--484.

\bibitem{Sen} Bl. Sendov, {\it Compulsory configurations of points in the plane}, Fundam. Appl. Math., 1:2 (1995), 491--516 (in Russian).

\bibitem{Nyk} H. Nyklova, {\it Almost empty polygons}, Studia Sci. Math. Hungar., 40(3) (2003), 269--286.

\bibitem{Kosh} V. A. Koshelev, {\it Interior Points in the Erdős--Szekeres Theorems}, Math. Notes, 91:4 (2012), 542–557.

\bibitem{FPM} V. A. Koshelev, {\it Almost empty hexagons}, J. Math. Sci., 164:1 (2010), 60–81.

\bibitem{h61_comp} V. A. Koshelev, {\it Computer Solution of the Almost Empty Hexagon Problem}, Math. Notes, 89:3 (2011), 455–458.

\bibitem{Brass_ref} P. Brass, {\it Empty monochromatic fourgons in two-colored point sets}, Geombinatorics, 14(2) (2004), 5--7.

\bibitem{Friedman_ref} E. Friedman, {\it 30 two-colored points with no empty monochromatic convex fourgons}, Geombinatorics, 14(2) (2004), 53--54.

\bibitem{Gulik_ref} R. Van Gulik, {\it 32 two-colored points with no empty monochromatic convex fourgons}, Geombinatorics, 15(1) (2005), 32--33.

\bibitem{HS_ref} C. Huemer, C. Seara, {\it 36 two-colored points with no empty monochromatic convex fourgons}, Geombinatorics, 19(1) (2009), 5--6.

\bibitem{h46} V. Koshelev, {\it On Erdős--Szekeres problem and related problems}, 2009, \url{https://arxiv.org/abs/0910.2700}

\bibitem{basu} D. Basu, K. Basu, B. B. Bhattacharya, S. Das, {\it Almost empty monochromatic triangles in planar point sets}, Discrete Appl. Math., 210 (2016), 207--213.

\bibitem{cravioto} J. Cravioto-Lagos, A. C. González-Martínez, T. Sakai, J. Urrutia, {\it On almost empty monochromatic triangles and convex quadrilaterals in colored point sets}, Graphs Combin., 35 (2019), 1475--1493.

\bibitem{nconv4} O. Aichholzer, T. Hackl, C. Huemer, F. Hurtado, B. Vogtenhuber, {\it Large bichromatic point sets admit empty monochromatic 4-gons}, SIAM J. Discrete Math., 23(4):2147–2155, 2010.


\bibitem{liu} L. Liu, Y. Zhang, {\it Almost empty monochromatic quadrilaterals in planar point sets}, Math. Notes, 103:3 (2018), 415--429.

\bibitem{z3} L. de Moura, N. Bjørner, {\it Z3: An efficient SMT solver}, Proc. TACAS (2008), 337--340.

\bibitem{clingo} M. Gebser, R. Kaminski, B. Schaub, M. Ostrowski, {\it Clingo = ASP + Control}, Technical Report, Univ. Potsdam, 2024.

\bibitem{minisat} N. Eén, N. Sörensson, {\it An extensible SAT-solver}, Proc. SAT (2003), 502--518.

\bibitem{glucose} G. Audemard, L. Simon, {\it Predicting learnt clauses quality in modern SAT solvers}, Proc. IJCAI (2009), 399--404.

\bibitem{kissat} A. Biere, {\it Kissat at the SAT Competition 2020}, SAT Competition (2020), 54.

\bibitem{Dehnhardt1987} H. Dehnhardt, {\it Leere konvexe Vielecke in ebenen Punktmengen}, Ph.D. Thesis, TU Braunschweig, 1987.

\bibitem{Aichholzer2002} O. Aichholzer, F. Aurenhammer, H. Krasser, {\it On the dual span of the order types}, Discrete Comput. Geom., 28 (2002), 467--484.

\bibitem{Aichholzer2013} O. Aichholzer, {\it The order type database}, http://tugraz.at, 2013.

\bibitem{BDV} A. Bialostocki, P. Dierker, B. Voxman, {\it Some notes on the Erdős--Szekeres theorem}, Discrete Mathematics, Vol. 91, No. 3, pp. 231--238, 1991.

\bibitem{Caro} Y. Caro, {\it On the generalized Erdős--Szekeres conjecture — a new upper bound}, Discrete Mathematics, Vol. 160, No. 1-3, pp. 229--233, 1996.

\bibitem{KPT} G. Károlyi, J. Pach, G. Tóth. {\it A modular version of the Erdős--Szekeres theorem}, Studia Scientiarum Mathematicarum Hungarica, Vol. 38, pp. 245--259, 2001.

\bibitem{mod} V. A. Koshelev, {\it The Erdős--Szekeres Theorem and Congruences}, Math. Notes, 87:4 (2010), 537–542.

\bibitem{BK_triangles} I. Bárány, G. Károlyi, {\it Problems and results around the Erdős--Szekeres theorem}, Discrete and Computational Geometry, Japanese Conference (JCDCG 2000), Lecture Notes in Computer Science, Vol. 2098, pp. 199--205, 2001.

\bibitem{BS_general} C. Bautista-Santiago, J. Cano, R. Fabila-Monroy, C. Hidalgo Toscano, C. Huemer, J. Leaños, T. Sakai, J. Urrutia, {\it Ramsey numbers for empty convex polygons}, EuroCG. Ljubljana, Slovenia, March 16--18, 2015.

\bibitem{K_arcs} Gy. Károlyi, J. Pach, G. Tóth, P. Valtr, {\it Ramsey-type results for geometric graphs, II}, Discrete Comput. Geom. 20(3) (1998), 375–388.

\bibitem{B_ordered} M. Balko, J. Cibulka, K. Král, J. Kynčl. {\it Ramsey numbers of ordered graphs}, Electron. J. Combin., 27:P1.16, 2020.

\end{thebibliography}
\end{document}